\documentclass[11pt]{amsart}

\usepackage[mathscr]{eucal}
\usepackage{amsmath,amssymb,amsfonts,amsthm,enumerate}

\newtheorem{theorem}{Theorem}[section]

\newtheorem{lemma}[theorem]{Lemma}

\newtheorem{proposition}[theorem]{Proposition}

\newcommand{\Z}{\mathbb Z}

\DeclareMathOperator{\ord}{ord}

\DeclareMathOperator{\supp}{Supp}

\newcommand{\la}{\langle}
\newcommand{\ra}{\rangle}
\newcommand{\be}{\begin{equation}}
\newcommand{\ee}{\end{equation}}
\newcommand{\und}{\;\mbox{ and }\;}

\newcommand{\ber}{\begin{eqnarray}}
\newcommand{\eer}{\end{eqnarray}}

\newcommand{\Summ}[1]{\underset{#1}{\sum}}

\newcommand{\Fc}{\mathcal F}
\newcommand{\vp}{\mathsf v}
\newcommand{\h}{\mathsf h}

\DeclareSymbolFont{goo}{OMS}{cmsy}{b}{n}
\DeclareMathSymbol{\gooT}{\mathalpha}{goo}{"1}
\newcommand{\bdot}{\mathbin{\gooT}}
\begin{document}

\title{Inverse Zero-Sum Problems III: Addendum}
\author{David J. Grynkiewicz}

\begin{abstract}
The Davenport constant for a finite abelian group $G$ is the minimal length $\ell$ such that any sequence of $\ell$ terms from $G$ must contain a nontrivial zero-sum sequence.  For the group $G=(\Z/n\Z)^2$, its value is $2n-1$, which is a classical result of Olson \cite{Olson-rk2}. The associated inverse question is to characterize those sequences of maximal length $2n-2$ that do not have a nontrivial zero-sum subsequence. A simple argument shows this to be equivalent to characterizing all  minimal zero-sum sequences of maximal length $2n-1$, with a minimal zero-sum sequence being one that cannot have its terms partitioned into two proper, nontrivial zero-sum subsequences. This was done in a series of papers  \cite{Reiher-propB-thesis} \cite{Gao-Ger-propB} \cite{propB-GGG} \cite{Prop-Schmid} \cite{Schlage-case9-propB}. However, there is a missing case in
\cite[Proposition 4.2]{propB-GGG}, leading to a missing case not treated in the portion of the proof covered in \cite{propB-GGG}. Both these deficiencies are corrected here using the methods of \cite{propB-GGG}. The correction contained in this manuscript will be incorporated into a larger, forthcoming work that encompasses the entirety of the characterization, and is posted in advance so that the corrected argument is made available to the research community before this larger project is completed.
\end{abstract}

\maketitle

\section{Introduction}\label{sec-intro}
Regarding combinatorial notation for sequences and subsums, we utilize the standardized system surrounding multiplicative strings as outlined in the references \cite{Ger-book} \cite{Gr-book}. For the reader new to this notational system, a complete summary of all relevant definitions and notation is given in Section \ref{Sec-Prelim}.

Let $G$ be a finite abelian group. The Davenport Constant $\mathsf D(G)$ for $G$ is the minimal integer $\ell$ such that any sequence of terms from $G$ with length at least $\ell$ must contain a nontrivial zero-sum subsequence, so $|S|\geq \mathsf D(G)$ implies $0\in \Sigma(S)$.  It is one of the most important zero-sum invariants in Combinatorial Number Theory. For a rank two abelian group $G=(\Z/m\Z)\times (\Z/n\Z)$, where $m\mid n$, we have $\mathsf D(G)=m+n-1$ by a classical result of Olson \cite{Olson-rk2} \cite[Theorem 5.8.3]{Ger-book}. The associated inverse question is to characterize those maximal length  sequences that do not have a nontrivial zero-sum subsequence.

If $S$ is a \textbf{zero-sum free} sequence of terms from an abelian group $G$, meaning $S$ has no proper, nontrivial zero-sum subsequence, then it is readily noted that extending the sequence by concatenating $S$ with the term $-\sigma(S)$ results in a minimal zero-sum  $S\bdot -\sigma(S)$, where a \textbf{minimal zero-sum} sequence is a zero-sum sequence that cannot have its terms partitioned into two proper, nontrivial zero-sum subsequences, while removing a term from a minimal zero-sum sequence results in a zero-sum free sequence.  As such, characterizing all maximal length zeros-sum free sequences is equivalent to characterizing all maximal length minimal zero-sum sequences. For the group  $G=(\Z/m\Z)\times (\Z/n\Z)$, where $m\mid n$, this was accomplished in a series of papers \cite{Reiher-propB-thesis} \cite{Gao-Ger-propB} \cite{propB-GGG} \cite{Prop-Schmid} \cite{Schlage-case9-propB}. However,  there is a missing case in
\cite[Proposition 4.2]{propB-GGG}, leading to a missing case not treated in the portion of the proof covered in \cite{propB-GGG}. The goal of this manuscript is to correct this
deficiency  by using the methods of \cite{propB-GGG} to handle the missing case. The correction contained in this manuscript will be incorporated into a larger, forthcoming work that encompasses the entirety of the characterization, and is provided here so that the corrected argument is made available to the research community before this larger project is completed.

Let $G=(\Z/n\Z)^2$. A basis for $G$ is a pair of elements $e_1,e_2\in G$ such that $\la e_1,e_2\ra=\la e_1\ra\oplus \la e_2\ra=(\Z/n\Z)\oplus (\Z/n\Z)=G$. For an integer $k\geq n$, we let $\mathsf s_{\leq k}(G)$ denote the minimal integer $\ell$ such that any sequence $S$ of terms from $G$ with length at least $\ell$ contains a nontrivial zero-sum of length at most $k$, so $|S|\geq \mathsf s_{\leq k}(G)$ implies $0\in \Sigma_{\leq k}(S)$. As part of Olson's proof that $\mathsf D((\Z/n\Z)^2)=2n-1$, he also showed that $\mathsf s_{\leq n}((\Z/n\Z)^2)=3n-2$ \cite{Olson-rk2} \cite[Theorem 5.8.3]{Ger-book}, and the inverse zero-sum question for the related constant $\mathsf s_{\leq n}((\Z/n\Z)^2)$ will play an important role in the proof of the inverse question for $\mathsf D((\Z/n\Z)^2)$.

To describe the characterization, we say that a sequence $S$ of terms from $G=(\Z/n\Z)^2$ has \textbf{Property A} if there is an (ordered) basis $(e_1,e_2)$ for $G$ such that $$\supp(S)\subseteq \{e_1\}\cup \big(\la e_1\ra+e_2\big).$$ Note that if $S$ has Property A with respect to the basis $(e_1,e_2)$, then it also satisfies Property A with respect to the basis $(e_1,e'_2)$ for any $e'_2\in \la e_1\ra+e_2$, so the role of $e_2$ is to define the coset $\la e_1\ra+e_2$. If $S$ is a zero-sum sequence of terms from $G$ satisfying Property A, then the number of terms from $\la e_1\ra+e_2$ must be congruent to $0$ modulo $n$. As such, if $|S|=n$, then either $$S=e_1^{[n]}\quad\mbox{ or } \quad S={\prod}^\bullet_{i\in [1,n]}(x_ie_1+e_2),$$
where $x_1,\ldots,x_n\in [0,n-1]$ with $x_1+\ldots+x_n\equiv 0\mod n$. Likewise, if $|S|=2n-1$, then \be\label{char-propB} S=e_1^{[n-1]}\bdot {\prod}^\bullet_{i\in [1,n]}(x_ie_1+e_2), \ee where $x_1,\ldots,x_n\in [0,n-1]$ with $x_1+\ldots+x_n\equiv 1\mod n$. Note, in the latter case, this forces there to be distinct terms from $\la e_1\ra+e_2$ contained in $S$, else we obtain the contradiction $x_1+\ldots+x_n=nx_1\equiv 0\not\equiv 1\mod n$.

It is rather immediate that any sequence satisfying \eqref{char-propB} is a minimal zero-sum of maximal length $2n-1$. If the converse to this statement holds for $G=(\Z/n\Z)^2$, that is, every minimal zero-sum sequence  with length $2n-1$ must have the form given in \eqref{char-propB} with respect to some basis $(e_1,e_2)$, then we say \textbf{Property B} holds for  $G=(\Z/n\Z)^2$. Per the discussion in the previous paragraph, this is equivalent to every minimal zero-sum  with maximal length $2n-1$ having Property A. In particular, Property B holding for $G=(\Z/n\Z)^2$ implies that any minimal zero-sum sequence with maximal length $2n-1$ must contain a term with multiplicity $n-1$. This is often how Property B was defined in earlier literature, for a fairly basic argument shows that if $S$ is a minimal zero-sum with $|S|=2n-1$ and $e_1$ is a term of $S$ with multiplicity at least $n-1$, then $S$ satisfies Property A (and thus has the form given in \eqref{char-propB}) with respect to some basis $(e_1,e_2)$, where we emphasize that the first element $e_1$ occurring in the basis $(e_1,e_2)$ is the same element $e_1$ that occurs with multiplicity $n-1$ in $S$.
Indeed, if $\varphi:G\rightarrow G/\la e_1\ra$ is the reduction modulo $\la e_1\ra$ homomorphism, then $G/\la e_1\ra\cong \Z/n\Z$, the subsequence of $\varphi(S)$ consisting of nonzero terms must be a minimal zero-sum sequence of length $n$ in $\varphi(G)\cong \Z/n\Z$, and the characterization of such sequences is known: they consist of a single order $n$ element repeated $n$ times \cite[Theorem 11.1]{Gr-book}. Also worth noting, if $e_1$ and $e_2$ are distinct terms that both occur with multiplicity $n-1$ in the minimal zero-sum sequence $S$ with $|S|=2n-1$, then $S=e_1^{[n-1]}\bdot e_2^{[n-1]}\bdot (e_1+e_2)$ with $(e_1,e_2)$ a basis with respect to which $S$ satisfies Property A.

Let $G=(\Z/n\Z)^2$. As noted earlier, the maximal length of a sequence $S$ of terms from $G$ with $0\notin \Sigma_{\leq n}(S)$ is $|S|=3n-3$. If every sequence $S$ of terms from $G$ with $0\notin \Sigma_{\leq n}(S)$ and $|S|=3n-3$
 must have the form $S=e_1^{[n-1]}\bdot e_2^{[n-1]}\bdot e_3^{[n-1]}$ for some $e_1,e_2,e_3\in G$, then we say that $G=(\Z/n\Z)^2$ has \textbf{Property C}. Note $0\notin \Sigma_{\leq n}(S)$ ensures $\ord(e_1)=\ord(e_2)=\ord(e_3)=n$. Thus there is some $f_1\in G$ such that $(f_1,e_2)$ is a basis for $G$. Letting $e_1=xf_1+ye_2$, we see that $(e_1,e_2)$ is a basis for $G$ unless $\gcd(x,n):=n/h>1$. However, if this were the case, then $T=e_1^{[h]}\bdot e_2^{[xh]}$ is a zero-sum subsequence of $S$ for some $x\in [0,\frac{n}{h}-1]$ having length $|T|= h+xh\leq n$, contradicting that $0\notin \Sigma_{\leq n}(S)$. Therefore  $(e_1,e_2)$ is a basis for $G$, and likewise  $(e_1,e_3)$ and $(e_2,e_3)$ must also be bases for $G$.

 There are two important and fairly nontrivial facts regarding Property C. The first is that if Property B holds for $G$, then property C holds for $G$. This is shown in \cite{Gao-Ger-propB}.
 The second is that any sequence $S$ of terms from $G$  with $0\notin \Sigma_{\leq n}(S)$ and $|S|=3n-3$ that satisfies the structural condition given in Property C must also satisfy Property A. What this means is that, if $G$ satisfies Property C, then the sequences $S$ of terms from $G$ with $0\notin \Sigma_{\leq n}(S)$ and $|S|=3n-3$ are characterized precisely as those for which there exists a basis $(e_1,e_2)$ for $G$ and $x\in [1,n-1]$ such that \be\label{PropC-char} S=e_1^{[n-1]}\bdot e_2^{[n-1]}\bdot (xe_1+e_2)^{[n-1]}.\ee The argument for this second fact, showing how the weaker condition given in Property C is nonetheless sufficient to derive the  precise structural characterization given above, can be found in several locations. It first appears (for $n$ prime) in \cite{emdeBoas-propC}. The general caes  appears as a special case of \cite[Theorem 3.1, Lemma 4.4]{Wolfgang-propC}, and can also be derived in a few pages by using the simplified core of  arguments from  \cite{Gr-uzi}, as shown in the forthcoming paper \cite{Gr-CL}.

\section{Notation and Preliminaries}\label{Sec-Prelim}

In  this section, we summarize the standardized notation for sequences and subsums as well as several lemmas from \cite{propB-GGG} that we will need to utilize in the main proof. In this paper, all intervals will be discrete, so for $x,\,y\in \Z$, we have $[x,y]=\{z\in \Z:\; x\leq z\leq y\}$.

Let $G$ be an abelian group. Following the tradition of Combinatorial Number Theory, a \textbf{sequence} of terms from $G$ is a finite, unordered string of elements from $G$.
 The term multi-set would also be appropriate but is not traditionally used. We let $\Fc(G)$ denote the free abelian monoid with basis $G$, which consists of all (finite and unordered) sequences $S$ of terms from $G$ written as multiplicative strings using the boldsymbol $\bdot$ . Thus a sequence $S\in \Fc(G)$ has the form  $$S=g_1\bdot\ldots\bdot g_\ell$$ with $g_1,\ldots,g_\ell\in G$ the terms in $S$.
Then $$\vp_g(S)=|\{i\in [1,\ell]:\; g_i=g\}|$$ denotes the multiplicity of the terms $g$ in $S$, allowing us to also represent a sequence $S$ in the form $$S={\prod}^\bullet_{g\in G}g^{[\vp_g(S)]},$$ where $g^{[n]}={\underbrace{g\bdot\ldots\bdot g}}_n$ denotes a sequence consisting of the term $g\in G$ repeated $n\geq 0$ times.
 The maximum multiplicity of a term of $S$ is the height of the sequence, denoted $$\h(S)=\max\{\vp_g(S):\; g\in G\}.$$ The support of the sequence $S$ is the subset of all elements of $G$ that are contained in $S$, that is, that occur with positive multiplicity in $S$, which is denoted $$\supp(S)=\{g\in G:\; \vp_g(S)>0\}.$$
 The length of the sequence $S$ is $$|S|=\ell=\Summ{g\in G}\vp_g(S).$$  A sequence $T\in \Fc(G)$ with $\vp_g(T)\leq \vp_g(S)$ for all $g\in G$ is called a subsequence of $S$, denoted $T\mid S$, and in such case, $S\bdot T^{[-1]}=T^{[-1]}\bdot S$ denotes the subsequence of $S$ obtained by removing the terms of $T$ from $S$, so $\vp_g(S\bdot T^{[-1]})=\vp_g(S)-\vp_g(T)$ for all $g\in G$.

 Since the terms of $S$ lie in an abelian group, we have the following notation regarding subsums of terms from $S$. We let $$\sigma(S)=g_1+\ldots+g_\ell=\Summ{g\in G}\vp_g(S)g$$ denote the sum of the terms of $S$ and call $S$ a zero-sum sequence when $\sigma(S)=0$. For $n\geq 0$, let
 \begin{align*}&\Sigma_n(S)=\{\sigma(T):\; T\mid S, \; |T|=n\},\quad
 \Sigma_{\leq n}(S)=\{\sigma(T):\; T\mid S, \; 1\leq |T|\leq n\},\quad\und\quad\\
 &\Sigma(S)=\{\sigma(T):\; T\mid S, \; |T|\geq 1\}
\end{align*}
 denote the variously restricted collections of  subsums of $S$. Finally, if $\varphi:G\rightarrow G'$ is a map, then $$\varphi(S)=\varphi(g_1)\bdot\ldots\bdot \varphi(g_\ell)\in \Fc(G')$$ denotes the sequence of terms from $G'$ obtained by applying $\varphi$ to each term from $S$.

 We will need the following basic lemmas (Lemmas 3.1, 3.2 and 3.3 from \cite{propB-GGG}) regarding perturbation of minimal zero-sum sequences satisfying \eqref{char-propB}. To this end, for a group $G=(\Z/m\Z)^2$, let $\Upsilon(G)$ be the set of all sequences $S$ satisfying \eqref{char-propB} (with $n=m$) with respect to some basis $(e_1,e_2)$ for $G$, and partition $$\Upsilon(G)=\Upsilon_u(G)\cup \Upsilon_{nu}(G),$$ where  $\Upsilon_u(G)\subseteq \Upsilon(G)$ contains all such sequences that have a  \emph{unique} term with multiplicity $m-1$, and $\Upsilon_{nu}(G)\subseteq \Upsilon(G)$ contains all such sequences $S$ of the form $S=e_1^{[m-1]}\bdot e_2^{[m-1]}\bdot (e_1+e_2)$ for some basis $(e_1,e_2)$.

 \begin{lemma}\label{lemma-pertub-I}
 Let $m\geq 4$, let $G=(\Z/m\Z)^2$, let $g\in G$ and suppose $$S=f_1^{[m-1]}\bdot {\prod}^\bullet_{i\in [1,m]}(x_if_1+f_2)\in \Upsilon_{u}(G),$$ where $x_1,\ldots,x_m\in [0,m-1]$ with $(f_1,f_2)$ a basis for $G$.
 \begin{itemize}
\item[1.] If $S'=f_1^{[-2]}\bdot S\bdot (f_1+g)\bdot (f_1-g)\in \Upsilon(G)$, then $g=0$ and $S'=S$.
\item[2.] If $S'=f_1^{[-1]}\bdot (x_jf_1+f_2)^{[-1]}\bdot S\bdot (f_1+g)\bdot (x_jf_1+f_2-g)\in \Upsilon(G)$, for some $j\in [1,m]$, then $g\in \{0,\, (x_j-1)f_1+f_2\}$ and $S'=S$.
\item[3.] If $S'=(x_jf_1+f_2)^{[-1]}\bdot (x_kf_1+f_2)^{[-1]}\bdot S\bdot (x_jf_1+f_2+g)\bdot (x_kf_1+f_2-g)\in \Upsilon(G)$, for some distinct $j,\,k\in [1,m]$, then $g\in \la f_1\ra$.

 \end{itemize}
 \end{lemma}

 \begin{lemma}\label{lemma-pertub-II(weak)}
 Let $m\geq 4$, let $G=(\Z/m\Z)^2$, let $g\in G$ and suppose $$S=f_1^{[m-1]}\bdot f_2^{[m-1]}\bdot (f_1+f_2)\in \Upsilon_{nu}(G),$$ with $(f_1,f_2)$ a basis for $G$.
 \begin{itemize}
\item[1.] If $S'=f_1^{[-2]}\bdot S\bdot (f_1+g)\bdot (f_1-g)\in \Upsilon(G)$, then $g\in \la f_2\ra$.
\item[2.] If $S'=f_2^{[-2]}\bdot S\bdot (f_2+g)\bdot (f_2-g)\in \Upsilon(G)$, then $g\in \la f_1\ra$.
\item[3.] If $S'=f_1^{[-1]}\bdot f_2^{[-1]}\bdot S\bdot (f_1+g)\bdot (f_2-g)\in \Upsilon(G)$, then $g\in \{0,\, -f_1+f_2\}$ and $S'=S$.
\item[4.] If $S'=f_1^{[-1]}\bdot (f_1+f_2)^{[-1]}\bdot S\bdot (f_1+g)\bdot (f_1+f_2-g)\in \Upsilon(G)$,  then
    $g\in \la f_2\ra$.
\item[5.]If $S'=f_2^{[-1]}\bdot (f_1+f_2)^{[-1]}\bdot S\bdot (f_2+g)\bdot (f_1+f_2-g)\in \Upsilon(G)$,  then
    $g\in \la f_1\ra$.
 \end{itemize}
 \end{lemma}

 \begin{lemma}\label{lemma-pertub-III(strong)}
 Let $m\geq 4$, let $G=(\Z/m\Z)^2$, let $g\in G$ and suppose $$S=f_1^{[m-1]}\bdot f_2^{[m-1]}\bdot (f_1+f_2)\in \Upsilon_{nu}(G),$$ with $(f_1,f_2)$ a basis for $G$.
 \begin{itemize}
\item[1.] If $S'=f_1^{[-2]}\bdot S\bdot (f_1+g)\bdot (f_1-g)\in \Upsilon_{nu}(G)$, then $g=0$ and $S'=S$.
\item[2.] If $S'=f_2^{[-2]}\bdot S\bdot (f_2+g)\bdot (f_2-g)\in \Upsilon_{nu}(G)$, then $g=0$ and $S'=S$.
\item[3.] If $S'=f_1^{[-1]}\bdot f_2^{[-1]}\bdot S\bdot (f_1+g)\bdot (f_2-g)\in \Upsilon_{nu}(G)$, then $g\in \{0,\, -f_1+f_2\}$ and $S'=S$.
\item[4.] If $S'=f_1^{[-1]}\bdot (f_1+f_2)^{[-1]}\bdot S\bdot (f_1+g)\bdot (f_1+f_2-g)\in \Upsilon_{nu}(G)$,  then
    $g\in \{0,\, f_2\}$ and $S'=S$.
\item[5.]If $S'=f_2^{[-1]}\bdot (f_1+f_2)^{[-1]}\bdot S\bdot (f_2+g)\bdot (f_1+f_2-g)\in \Upsilon_{nu}(G)$,  then
    $g\in \{0,\,f_1\}$ and $S'=S$.
 \end{itemize}
 \end{lemma}

\section{The Missing Case}\label{sec-k=n-1}

The proof of \cite[Proposition 4.2]{propB-GGG} utilizes only that $S$ is a long zero-sum with $0\notin \Sigma_{\leq n-1}(S)$. As such, it does not account for the possibility given in Lemma \ref{lemma-casen-gen}.2, which is the sole example of a sequence not satisfying Property A that meets the requisite hypotheses.  
Specifically, the sequence given in Lemma \ref{lemma-casen-gen}.2 falls under Case 2.4 in the proof of \cite[Proposition 4.2]{propB-GGG} with $W_0=e_1^{[n-1]}\bdot (xe_1+e_2)^{[x^*]}\bdot e_2^{[n-x^*]}$ and $W_i=(xe_1+2e_2)\bdot (xe_1+e_2)^{[n-1-x^*]}\bdot e_2^{[x^*-1]}\bdot e_1$, where $xx^*\equiv 1\mod n$,  and there the problem lies in that the existence of $\nu'$ is not guaranteed as claimed in the proof.
We correct the proof and statement of
\cite[Proposition 4.2]{propB-GGG} in Lemma \ref{lemma-casen-gen} below,
which we prove using an  inductive argument
(with Lemma \ref{lemma-casen} serving as the base of the induction) rather than by the massive case analysis employed in \cite{propB-GGG}. The simplified argument makes it more evident where the missing case emerges from the argument.

\begin{lemma}\label{lemma-casen} Let $n\geq 2$,  let $G=(\Z/n\Z)^2$, and suppose Property B holds for $G$. If   $S\in\Fc(G)$  is a zero-sum sequence with $|S|=3n-1$ and $0\notin \Sigma_{\leq n-1}(S)$,  then  there is a basis $(e_1,e_2)$ for $G$ such that either
\begin{itemize}
\item[1.] $\supp(S)\subseteq \{e_1\}\cup \big(\la e_1\ra+e_2\big)$ and $\vp_{e_1}(S)\equiv -1\mod n$, or
  \item[2.] $S=e_1^{[n]}\bdot e_2^{[n-1]}\bdot (xe_1+e_2)^{[n-1]} \bdot (xe_1+2e_2)$ for some $x\in [2,n-2]$ with $\gcd(x,n)=1$.
\end{itemize}
\end{lemma}

\begin{proof}
By hypothesis, Property B holds for $G$, which implies Property C does as well, as discussed in Section \ref{sec-intro}.
Since $0\notin \Sigma_{\leq n-1}(S)$ and $\mathsf s_{\leq n}(G)=\mathsf s_{\leq n}((\Z/n\Z)^2)=3n-2$ (as noted in Section \ref{sec-intro}), it follows that any sequence $T\mid S$ with $|T|\geq 3n-2$ must contain an $n$-term zero-sum subsequence. Moreover, any subsequence $T\mid S$ with $|S|=3n-3$ that does not have the form given in \eqref{PropC-char}  must also contain an $n$-term zero-sum subsequence. Since $|S|\geq 3n-1\geq 3n-2$, $S$ contains at least one $n$-term zero-sum subsequence $W_1\mid S$.
If $W_1\mid S$ is any such subsequence, then $W_0=S\bdot W_1^{[-1]}$ will be a zero-sum sequence (as $S$ and $W_1$ are both  zero-sum) of length $|S|-|W_1|=2n-1$. If $W_0$ is not minimal, then it could be factored as a product of two nontrivial zero-sum subsequences, one of which must have length at most $n-1$ by the Pigeonhole Principle.
Since this would contradict the hypothesis $0\notin \Sigma_{\leq n-1}(S)$, it follows that $W_0$ is a minimal zero-sum sequence of length $2n-1$.
As such (in view of Property B), there is  a basis $(e_1,e_2)$ for $G$ such that property A holds for $W_0$, in which case $W_0$ satisfies \eqref{char-propB}. In particular, $\mathsf h(S)\geq \vp_{e_1}(S)\geq n-1$.  We call any decomposition $S=W_0\bdot W_1$ with $W_1$ and $W_0$ zero-sum subsequences of lengths $|W_0|=2n-1$ and $|W_1|=n$ a \emph{block decomposition} of $S$. As just noted, $S$ has a block decomposition.

Suppose $n=2$ and let $S=W_0\bdot W_1$ be any block decomposition.  Then $0\notin\Sigma_{\leq n-1}(S)$ forces $W_1$ to consist of a single non-zero term repeated twice, and $W_0$ being a minimal zero-sum of length $2n-1=3$  forces $W_0$ to consist of  the three nonzero elements of $G=(\Z/2\Z)^2$ each repeated once. In such case, Item 1 holds using any basis for $G=(\Z/2\Z)^2$, allowing us to assume $n\geq 3$.

The following observation will be  used in several arguments: \be\label{useful-obs} \mbox{if $g,\,h\in \supp(S)$ with $\vp_g(S)\geq n-2$ and $h\in \la g\ra$, then $h=g$}.\ee Indeed if $h=xg$ with $x\in [2,n]$, then $g^{[n-x]}\bdot h$ would be a  zero-sum subsequence of $S$ with length $n-x+1\leq n-1$, contradicting that $0\notin \Sigma_{\leq n-1}(S)$.

\subsection*{Case 1:} $\mathsf h(S)\geq n+1$.

In this case, there is some term $e_3\in\supp(S)$ with $\vp_{e_3}(S)\geq n+1$ and there is a block decomposition $S=W_0\bdot W_1$ with $W_1=e_3^{[n]}$ and $e_3\in \supp(W_0)$. Let $(e_1,e_2)$ be a basis such that Property A holds for  $W_0$, so $\supp(W_0)\subseteq \{e_1\}\cup \big(\la e_1\ra+e_2\big)$ with $W_0$ having the form given by \eqref{char-propB}. Since $W_1=e_3^{[n]}$ with $e_3\in \supp(W_0)$, Item 1 now follows using the basis $(e_1,e_2)$.

\subsection*{Case 2:} $\mathsf h(S)=n$.

Let $e_1\in \supp(S)$ be a term with $\vp_{e_1}(S)=n$. Then $S$ has block decomposition $S=W_0\bdot W_1$ with $$W_1=e_1^{[n]}.$$ Note $$\supp(W_0)\cap \supp(W_1)=\emptyset\quad\und\quad\ord(e_1)=n$$ in view of $0\notin \Sigma_{\leq n-1}(S)$ and the case hypothesis.  As  $W_0$ is a minimal zero-sum of maximal length with Property B holding,  there is some $g\in \supp(W_0)$ with $\vp_{g}(W_0)=\vp_{g}(S)=n-1$. We handle two subcases.

\subsection*{Case 2.1:} There are distinct $e_2,e_3\in \supp(W_0)$ with $\vp_{e_2}(W_0)=\vp_{e_3}(W_0)=n-1$.

In this case, $$W_0=e_2^{[n-1]}\bdot e_3^{[n-1]}\bdot (e_2+e_3)$$ with $(e_2,e_3)$ a basis for $G$.  Let $$e_1=xe_2+ye_3$$ with $x,\,y\in [0,n-1]$.

Since $\supp(W_0)\cap \supp(W_1)=\emptyset$, we have $e_1\notin\{e_2,e_3\}$, whence \eqref{useful-obs} ensures $x,\,y\in [1,n-1]$. If $y=1$, then Item 1 holds with basis $(e_2,e_3)$. If $x=1$, then Item 1  holds with basis $(e_3,e_2)$. Therefore we can assume $$x,\,y\in [2,n-1].$$  Consider the subsequence $T=e_1^{[n-1]}\bdot e_2^{[n-1]}\bdot e_3^{[n-1]}$ of $S$.

Suppose $0\in \Sigma_{\leq n}(T)$.  Then there exists a zero-sum subsequence $W'_1\mid T$ with $|W'_1|=n$ (in view of  the hypothesis $0\notin \Sigma_{\leq n-1}(S)$), and we obtain a block decomposition $S=W'_0\bdot W'_1$.
It is impossible for a zero-sum sequence of length $n$ to have a term with multiplicity exactly $n-1$. Thus $W'_1\mid e_1^{[n-2]}\bdot e_2^{[n-2]}\bdot e_3^{[n-2]}$, ensuring that \be\label{supp4}e_1,e_2,e_3,(e_2+e_3)\in \supp(W'_0),\ee
which are all distinct elements in view of $e_1=xe_2+ye_3$ with $x,y\in [2,n-1]$.
Let $(f_1,f_2)$ be a basis such that Property A holds for $W'_0$.  As $n\geq 3$ with $\vp_{e_1+e_2}(S)=1$, we must  have $f_1\in \{e_1,e_2,e_3\}$. If $f_1=e_1$, then \eqref{supp4} and Property A imply $e_2=(e_2+e_3)-e_3\in \la e_1\ra$ and $e_3=(e_2+e_3)-e_2\in \la e_1\ra$, yielding the contradiction $G=\la e_2,e_3\ra=\la e_1\ra$ (as $G$ is not cyclic). If $f_1=e_2$, then \eqref{supp4} and Property A imply $e_1-e_3=(xe_2+ye_3)-e_3\in \la e_2\ra$, forcing $y=1$. If $f_1=e_3$, then \eqref{supp4} and Property A imply $e_1-e_2=(xe_2+ye_3)-e_2\in \la e_3\ra$, forcing $x=1$. As both these cases were already handled, we can now assume $0\notin \Sigma_{\leq n}(T)$.

Since $0\notin \Sigma_{\leq n}(T)$ with $|T|=3n-3$ and $T=e_1^{[n-1]}\bdot e_2^{[n-1]}\bdot e_3^{[n-1]}$, Property C ensures that $T$ has the form given by \eqref{PropC-char}, as remarked in Section \ref{sec-intro}. As a result, depending on which element from $\{e_1,e_2,e_3\}$ plays the role of $e_1$ in \eqref{PropC-char}, it follows that either $e_1-e_2=(xe_2+ye_3)-e_2\in \la e_3\ra$ or $e_1-e_3=(xe_2+ye_3)-e_3\in \la e_2\ra$ or $$-e_2+e_3\in \la e_1\ra=\la xe_2+ye_3\ra.$$ The first case forces $x=1$ and the second forces $y=1$, both cases handled already. Since $(e_2,e_3)$ is a basis, the final case implies $\gcd(x,n)=\gcd(y,n)=1$ with $-e_2+e_3=y^*e_1$, where $y^*\in [1,n-1]$ is the multiplicative inverse of $y$ modulo $n$, in turn forcing $y\equiv -x\mod n$. Hence, since $x,y\in [2,n-1]$, we conclude that $x,y,y^*\in [2,n-2]$. But now,    $e_3=y^*e_1+e_2$ and $e_2+e_3=y^*e_2+2e_2$.  In this case,
$S=e_1^{[n]}\bdot e_2^{[n-1]}\bdot e_3^{[n-1]}\bdot (e_2+e_3)=e_1^{[n]}\bdot e_2^{[n-1]}\bdot (y^*e_1+e_2)^{[n-1]}\bdot (y^*e_1+2e_2)$, and Item 2 follows with basis $(e_1,e_2)$, completing Case 2.1.

\subsection*{Case 2.2:} There is a unique $g\in \supp(W_0)$  with $\vp_{g}(W_0)=n-1$.

Let $(f_1,f_2)$ be basis for which  Property A  holds for $W_0$.  Since $\supp(W_0)\cap \supp(W_1)=\emptyset$, we have $\vp_{f_1}(S)=\vp_{f_1}(W_0)=n-1$, while $f_1=g$ follows by case hypothesis.  Let $$e_1=xf_1+yf_2$$ and let  $$zf_1+f_2\in \supp(W_0)\setminus \{f_1\},\quad \mbox{where $z\in [0,n-1]$},$$ be arbitrary.
Then $T=S\bdot e_1^{[-1]}\bdot (zf_1+f_2)^{[-1]}$ is a subsequence $T\mid S$ with $|T|=3n-3$ which has only two elements with multiplicity equal to $n-1$. If follows from Property C that there is a zero-sum subsequence $W'_1\mid T$ with $|W'_1|=n$ (as $0\notin \Sigma_{\leq n-1}(S)$). Let $W'_0=S\bdot (W'_1)^{[-1]}$. Then $S=W'_0\bdot W'_1$ is a block decomposition with $$e_1,\,zf_1+f_2\in \supp(W'_0).$$ Moreover, since an $n$-term zero-sum sequence cannot contain a term with multiplicity exactly $n-1=\vp_{f_1}(S)$, it follows that $$f_1\in \supp(W'_0).$$
Note that the elements $e_1$, $f_1$ and $zf_1+f_2$ are distinct in view of $\supp(W_0)\cap \supp(W_1)=\emptyset$ with $(f_1,f_2)$ a basis.
Let $(e'_1,e'_2)$ be a basis for which Property A holds for $W'_0$. By case hypothesis, $e_1$ and $f_1$ are the only terms of $S$ with multiplicity at least $n-1$.
Thus either $e'_1=e_1$ or $e'_1=f_1$.
If $e'_1=f_1$, then $e_1,\,zf_1+f_2\in \supp(W'_0)$ combined with Property A  implies $e_1-(zf_1+f_2)=(xf_1+yf_2)-(zf_1+f_2)\in \la f_1\ra$, forcing $y=1$. In such case, Item 1 holds with basis $(f_1,f_2)$. On the other hand, if $e'_1=e_1$, then $f_1, zf_1+f_2\in \supp(W'_0)$ combined with Property A implies \be\label{trunch}(zf_1+f_2)-f_1\in \la e_1\ra=\la xf_1+yf_2\ra.\ee Since $(f_1,f_2)$ is basis, \eqref{trunch} is only possible if $\gcd(y,n)=1$, implying that $$z-1\equiv xy^*\mod n,$$ where $y^*\in [1,n-1]$ is the multiplicative inverse of $y$ modulo $n$.   Since
$|\supp(W_0)\setminus \{f_1\}|\geq 2$ per the remarks given in Section \ref{sec-intro} regarding minimal zero-sums of length $2n-1$ satisfying Property A, there is another $z'f_1+f_2\in \supp(W_0)\setminus \{f_1\}$ with $z'\in [0,n-1]$ and $z'\neq z$. Repeating the previous argument using $z'f_1+f_2$ in place of $zf_1+f_2$, we likewise conclude that $z'-1\equiv xy^*\equiv z-1\mod n$. But this forces $z'=z$, contradicting that $zf_1+f_2$ and $z'f_1+f_2$ are distinct, which completes Case 2.

\subsection*{Case 3:} $\mathsf h(S)=n-1$.

Suppose there is a unique term $g\in \supp(S)$ with $\vp_{g}(S)= n-1$, let  $S=W_0\bdot W_1$ be an arbitrary block decomposition, and let $(e_1,e_2)$ be any  basis with respect to which Property A holds for  $W_0$. Then  we must have $e_1=g$ as $g$ is the unique term with $\vp_{g}(S)\geq n-1$. Let $xe_1+ye_2\in \supp(W_1)$ and  $ze_1+e_2\in \supp(W_0)\setminus \{e_1\}$ be arbitrary.
Note $\vp_{e_1}(W_0)=\vp_{e_1}(S)=n-1$ ensures $e_1\neq xe_1+ye_2$.
Then $T=S\bdot (xe_1+ye_2)^{[-1]}\bdot (ze_1+e_2)^{[-1]}$ is a subsequence $T\mid S$ with $|T|=3n-3$ containing exactly one term with multiplicity $n-1$. It thus follows from Property C that there is a zero-sum subsequence $W'_1\mid T$ with $|W'_1|=n$. Setting $W'_0=S\bdot (W'_1)^{[-1]}$, we see that $S=W'_0\bdot W'_1$ is a block decomposition having $(xe_1+ye_2),\, (ze_1+e_2)\in \supp(W'_0)$.
Letting $(f_1,f_2)$ be a basis for which Property A holds for $W'_0$, it follows that $f_1=g=e_1$ as remarked earlier, and thus $(xe_1+ye_2)-(ze_1+e_2)\in \la e_1\ra$.
Hence $y=1$, and as $xe_1+ye_2\in \supp(W_1)$ was arbitrary, it follows that
$\supp(S)\subseteq \{e_1\}\cup \big(\la e_1\ra+e_2\big)$, yielding Item 1 with basis $(e_1,e_2)$.  So we can now assume there are at least two distinct elements from $\supp(S)$ having multiplicity $n-1$ in $S$.

Let $S=W_0\bdot W_1$ be a block decomposition of $S$ and let $(e_1,e_2)$ be a basis with respect to which Property A  holds for $W_0$. Then $\vp_{e_1}(S)=n-1$ and $$\supp(W_0)\subseteq e_1\cup \big(\la e_1\ra+e_2\big).$$ Since an $n$-term zero-sum cannot have a term with multiplicity exactly $n-1$,  it follows that any term $g\in \supp(S)$ with $\vp_{g}(S)=n-1$ must have $g\in \supp(W_0)$, and thus $g=e_1$ or $g=ze_1+e_2$ for some $z\in [0,n-1]$. As there are at least two terms of $S$ multiplicity $n-1$, there is at least one such term of the form $g=ze_1+e_2$, and then by replacing the basis $(e_1,e_2)$ with the basis $(e_1,ze_1+e_2)$, we can assume the basis $(e_1,e_2)$ for which
Property A holds for $W_0$ has  $\vp_{e_2}(S)=n-1$ (though we are \emph{not} assured that $\vp_{e_2}(W_0)=n-1$). We handle two further subcases.

\subsection*{Case 3.1:}  There is a third term $e_3=ze_1+e_2$ with $z\in [1,n-1]$ that also has $\vp_{e_3}(S)=n-1$.

In this case, since $S$ is zero-sum, we have $$S=e_1^{[n-1]}\bdot e_2^{[n-1]}\bdot (ze_1+e_2)^{[n-1]}\bdot (xe_1+ye_2)\bdot ((1-x+z)e_1+(2-y)e_2),$$ for some $x,y\in [0,n-1]$. If $y=1$, then $2-y=1$ and  Item 1 holds with basis $(e_1,e_2)$.
In view of $\vp_{e_1}(S)=n-1$ and \eqref{useful-obs}, we cannot have $y=0$. Therefore we can assume $$y\in [2,n-1].$$ If there were a fourth term with multiplicity $n-1$ in $S$, then this would require $n=3$, so that $y\in [2,n-1]$ forces $y=2$, and would also require $(xe_1+ye_2)=(1-x+z)e_1+(2-y)e_2$, so that
$2=y\equiv 2-y= 0\mod n$, which is not possible. Therefore we find that $e_1$, $e_2$ and $ze_1+e_2$ are the only terms with multiplicity $n-1$ in $S$.

Consider  $T=S\bdot (xe_1+ye_2)^{[-1]}$. Since $|T|=3n-2$, there is a zero-sum subsequence $W'_0\mid T$ with $|W'_0|=n$. Then $S=W'_0\bdot W'_1$ is a block decomposition, where $W'_0=S\bdot (W'_1)^{[-1]}$, with \be\label{star}e_1,e_2,ze_1+e_2,xe_1+ye_2\in \supp(W'_0).\ee Note $e_1,e_2,ze_1+e_2\in \supp(W'_0)$ since the $n$-term zero-sum $W'_1$ cannot contain a term with multiplicity exactly $n-1$.
If $(e'_1,e'_2)$ is a basis for which Property A holds for $W'_0$, then we must have $e'_1\in \{e_1,e_2,ze_1+e_2\}$ as these are the only terms with multiplicity $n-1$ as noted earlier in the subcase.
If $e'_1=e_1$, then Property A and \eqref{star} imply $(xe_1+ye_2)-(ze_1+e_2)\in \la e_1\ra$, yielding the contradiction  $y=1$. If $e'_1=e_2$, then Property A and \eqref{star} imply $(xe_1+ye_2)-(ze_1+e_2)\in \la e_2\ra$ and $(ze_1+e_2)-e_1\in \la e_2\ra$, whence $x=z=1$, and now Item 1 holds with basis $(e_2,e_1)$. Finally,  if $e'_1=ze_1+e_2$, then Property A and \eqref{star} imply $-e_1+e_2\in \la ze_1+e_2\ra$ and $(xe_1+ye_2)-e_2\in \la  ze_1+e_2\ra$, whence $z=n-1$ and $x\equiv (y-1)z\equiv 1-y\mod n$.
In this case, $$S=e_1^{[n-1]}\bdot e_2^{[n-1]}\bdot (-e_1+e_2)^{[n-1]}\bdot (xe_1+(1-x)e_2)\bdot ( -xe_1+(x+1)e_2),$$ and now Item 1 holds with basis $(-e_1+e_2,e_2)$, completing the subcase.

\subsection*{Case 3.2:}  $e_1,\,e_2\in \supp(S)$ are the only elements with multiplicity $n-1$ in $S$.

Since $|\supp(W_0)\setminus \{e_1\}|\geq 2$ (per the remarks in Section \ref{sec-intro} regarding minimal zero-sum sequences of length $2n-1$ satisfying Property A), there must be some $ze_1+e_2\in \supp(W_0)$ with $z\in [1,n-1]$. Let $ze_1+e_2\in\supp(W_0)$ be an arbitrary such element, and let $xe_1+ye_2\in \supp(W_1)$ be arbitrary.
Since $\vp_{e_1}(S)=\vp_{e_1}(W_0)=n-1$, \eqref{useful-obs} ensures that $y\neq 0$, so $$y\in [1,n-1].$$
Now $T=S\bdot (ze_1+e_2)^{[-1]}\bdot (xe_1+ye_2)^{[-1]}$ is a subsequence with $|T|=3n-3$ and only two terms having multiplicity $n-1$. Thus Property C implies that there is a zero-sum subsequence $W'_1\mid T$ with $|W'_1|=n$. Letting $W'_0=S\bdot (W'_1)^{[-1]}$, we have a block decomposition $S=W'_0\bdot W'_1$ with
\be\label{eaglete}xe_1+ye_2,ze_1+e_2,e_1,e_2\in \supp(W'_0).\ee Note $e_1,\,e_2\in \supp(W'_0)$ since an $n$-term zero-sum sequence cannot have a term with multiplicity exactly $n-1$, and the terms listed in \eqref{eaglete} are distinct apart from the possible equalities $xe_1+ye_2=e_2$ or $xe_1+ye_2=ze_1+e_2$. If $(e'_1,e'_2)$ is a basis for which Property A holds for $W'_0$, then we must have $e'_1=e_1$ or $e'_1=e_2$, as these are the only terms having multiplicity at least $n-1$. If $e'_1=e_1$, then Property A and \eqref{eaglete} imply $(xe_1+ye_2)-(ze_1+e_2)\in \la e_1\ra$, whence $y=1$. If $xe_1+ye_2=e_2$, then $y=1$ as well. If $e'_1=e_2$ and $xe_1+ye_2\neq e_2$, then Property A and \eqref{eaglete} imply $(xe_1+ye_2)-e_1\in \la e_2\ra$ and $(xe_1+ye_2)-(ze_1+e_2)\in \la e_2\ra$, whence $x=z=1$.

As just shown, an arbitrary $xe_1+ye_2\in \supp(W_1)$ has $x=1$ or $y=1$. If we always have $y=1$, then Item 1 follows with basis $(e_1,e_2)$. Therefore we can assume there is some $$e_1+ye_2 \in \supp(W_1)\quad\mbox{ with $y\in [2,n-1]$}$$ (note that $y\neq 1$ ensures that the coefficient of $e_1$ is one, as just noted), and then the previous argument forces $z=1$. This must then be the case for an arbitrary $ze_1+e_2\in \supp(W_0)$ with $z\in [1,n-1]$, whence $\supp(W_0)=\{e_1,e_2,e_1+e_2\}$ and $W_0=e_1^{[n-1]}\bdot e_2^{[n-1]}\bdot (e_1+e_2)$. But now we cannot have $\supp(W_1)\subseteq \{e_2\}\cup \big(\la e_2\ra+e_1\big)$, lest Item 1 holds with basis $(e_2,e_1)$, so there must be some $$xe_1+e_2\in \supp(W_1)\quad\mbox{ with $x\in [2,n-1]$}$$ (as before, $x\neq 1$ ensures the coefficient of $e_2$ is one).
Letting  $T'=S\bdot (xe_1+e_2)^{[-1]}\bdot (e_1+ye_2)^{[-1]}$, we find that $|T|=3n-3$ with $T$ containing exactly two terms with multiplicity $n-1$. Thus Property C  implies that there is a zero-sum subsequence $W''_1\mid T$ with $|W''_1|=n$. Letting $W''_0=S\bdot (W''_1)^{[-1]}$, we find that $S=W''_0\bdot W''_1$ is a block decomposition with $xe_1+e_2,e_1+ye_2\in \supp(W''_0)$. If $(e''_1,e''_2)$ is a basis for which Property A holds for $W''_0$, then we must have $e''_1=e_1$ or $e''_1=e_2$, as these are the only terms with multiplicity at least $n-1$ in $S$.  In the former case, Property A gives $(xe_1+e_2)-(e_1+ye_2)\in \la e_1\ra$, yielding the contradiction $y=1$. In the latter case, Property A gives
$(xe_1+e_2)-(e_1+ye_2)\in \la e_2\ra$, yielding the contradiction $x=1$. As this exhausts all possibilities, the case and proof are complete.
\end{proof}

\begin{lemma}\label{lemma-casen-gen} Let $n\geq 2$, let $s\geq 1$,   let $G=(\Z/n\Z)^2$, and suppose Property B holds for $G$. If   $S\in\Fc(G)$  is a zero-sum sequence with $|S|=(2+s)n-1$ and $0\notin \Sigma_{\leq n-1}(S)$, then  there is a basis $(e_1,e_2)$ for $G$ such that either
\begin{itemize}
\item[1.] $\supp(S)\subseteq \{e_1\}\cup \big(\la e_1\ra+e_2\big)$ and $\vp_{e_1}(S)\equiv -1\mod n$, or
  \item[2.] $S=e_1^{[an]}\bdot e_2^{[bn-1]}\bdot (xe_1+e_2)^{[cn-1]} \bdot (xe_1+2e_2)$ for some $x\in [2,n-2]$ with $\gcd(x,n)=1$, and some  $a,b,c\geq 1$ with $a+b+c=2+s$.
\end{itemize}
\end{lemma}

\begin{proof}
If  $s=1$, the desired result follows by Lemma \ref{lemma-casen}. Thus we may assume $s\geq 2$ and proceed by induction on $s$.
As in the proof of Lemma \ref{lemma-casen}, since Property B holds for $G$ by hypothesis, this implies Property C does as well, with all the relevant commentary from Section \ref{sec-intro} applicable.
Since $0\notin \Sigma_{\leq n-1}(S)$ and $\mathsf s_{\leq n}(G)=3n-2$, it follows that any sequence $T\mid S$ with $|T|\geq 3n-2$ must contain an $n$-term zero-sum subsequence. Moreover, any subsequence $T\mid S$ with $|S|=3n-3$ that does not have the form given by \eqref{PropC-char} must also contain an $n$-term zero-sum subsequence. Since $|S|=(s-1)n+3n-1$, $S$ contains at least $s$ disjoint  $n$-term zero-sum subsequences $W_1\bdot\ldots\bdot W_s\mid S$, so  $W_i$ is zero-sum and $|W_i|=n$ for $i\in [1,s]$. Then $R=W_1\bdot \ldots \bdot W_s$ is a zero-sum subsequence of length $sn$. If $R\mid S$ is any zero-sum sequence of length $tn$ with $t\in [2,s]$, then, since $|R|=(t-3)n+3n$, we find at least $t-2$ disjoint zero-sum subsequences $W_1\bdot\ldots \bdot W_{t-2}\mid R$ with $|W_i|=n$ for  $i\in [1,t-2]$. But then $R\bdot (W_1\bdot\ldots \bdot W_{t-2})^{[-1]}$ is a zero-sum of length $2n>\mathsf D(G)$ with $0\notin\Sigma_{\leq n-1}(S)$, which must then factor as a product of two zero-sum subsequences $W_{t-1}$ and $W_t$ each of length $n$. Thus any zero-sum subsequence
$R\mid S$ of length $tn$, with $t\in [1,s]$, factors $R=W_1\bdot\ldots \bdot W_{t}$ as a product of $t$ zero-sum subsequences of length $n$ (this is trivial for $t=1$). Returning to the case $t=s$ and setting $W_0=S\bdot (W_1\bdot\ldots\bdot W_s)^{[-1]}$, we find that $W_0$ is a zero-sum subsequence of length $|W_0|=2n-1$, and since $0\notin \Sigma_{\leq n-1}(S)$, it must be a minimal zero-sum sequence. Thus, as property B holds for $G$, it follows that Property A holds for $W_0$.
Define a \emph{block decomposition} of $S$ to be a factorization $S=W_0\bdot W_1\bdot\ldots W_s$ into zero-sum subsequences $W_i$ for $i\in [0,s]$ such that $|W_0|=2n-1$ and $|W_i|=n$ for $i\in [1,s]$.  As just noted, $S$ has a block decomposition.
As in Lemma \ref{lemma-casen}, if $n=2$, then $W_0$ must consist of the three non-zero element of $G=(\Z/2\Z)^2$ with each $W_i$, for $i\in [1,s]$, consisting of a single non-zero term repeated twice. In such case, Item 1 holds using any basis $(e_1,e_2)$ of $G$, allowing us to assume $n\geq 3$.

\subsection*{Case 1:} There is a zero-sum subsequence $S'\mid S$ with $|S|=3n-1$ that satisfies Lemma \ref{lemma-casen}.2.

In this case, \be\label{S'-define}S'=e_1^{[n]}\bdot e_2^{[n-1]}\bdot (xe_1+e_2)^{[n-1]}\bdot (xe_1+2e_2)\ee for some $x\in [2,n-2]$ with $\gcd(x,n)=1$ and $(e_1,e_2)$ a basis for $G$ (implying $n\geq 5$). Set $$e_3=xe_1+e_2$$ and note that the above  is equivalent to
$S'=e_1^{[n]}\bdot e_2^{[n-1]}\bdot e_3^{[n-1]}\bdot (e_2+e_3)$ and $$e_1=x^*(e_2-e_3)$$ with $(e_2,e_3)$ a basis for $G$, where $x^*\in [2,n-2]$ is the multiplicative inverse of $-x$ modulo $n$, from which it is evident that there is complete symmetry between the elements $e_2$ and $e_3$.  Set $$W_0=e_2^{[n-1]}\bdot e_3^{[n-1]}\bdot (e_2+e_3)\quad\und\quad W_1=e_1^{[n]}=(x^*e_2-x^*e_3)^{[n]}.$$ Now $S\bdot (S')^{[-1]}$ is a zero-sum subsequence of length $(s-1)n$, thus factoring into a product of $s-1$ zero-sum subsequences of length $n$ as argued earlier. It follows that there is a block decomposition $S=W_0\bdot W_1\bdot W_2\bdot \ldots\bdot W_s$.
Observe that any zero-sum sequence $T$ with $\supp(T)\subseteq \{e_1,e_2,e_3\}$ and $|T|=n$ must have the form $T=e_i^{[n]}$ for some $i\in [1,3]$. Hence, if we have $\supp(W_2\bdot\ldots\bdot W_s)\subseteq \{e_1,e_2,e_3\}$, then Item 2 follows. We can thus w.l.o.g. assume there is some \be\label{gnotthere}g\in \supp(W_2)\quad\mbox{ with $g\notin \{e_1,e_2,e_3\}$}.\ee Apply Lemma \ref{lemma-casen} to $W_0\bdot W_2$.


\subsection*{Subcase 1.1:} Lemma \ref{lemma-casen}.2 holds for $W_0\bdot W_2$.

Then $$W_0\bdot W_2=(e'_1)^{[n]}\bdot (e'_2)^{[n-1]}\bdot (e'_3)^{[n-1]}\bdot (e'_2+e'_3)$$ for some $e'_1,e'_2,e'_3\in G$ with $(e'_1,e'_2)$ a basis and $$e'_3=ye'_1+e'_2\quad\mbox{ for some $y\in [2,n-2]$ with $\gcd(y,n)=1$}.$$ In particular, $|\supp(W_0\bdot W_2)|=4$ and $\mathsf h(W_0\bdot W_2)=n$.
If $\vp_{e_2+e_3}(W_0\bdot W_2)=n$, then $e_2+e_3$ would have multiplicity exactly $\vp_{e_2+e_3}(W_0\bdot W_2)-\vp_{e_2+e_3}(W_0)=n-1$ in the $n$-term zero-sum $W_2$, which is not possible. Therefore $\vp_{e_2+e_3}(W_0\bdot W_2)\leq n-1$.

Let us next handle the case when $\vp_{e_2+e_3}(W_0\bdot W_2)=n-1$. In such case, $\vp_{e_2+e_3}(W_2)=n-2$ and the only possible terms that can have multiplicity $n\geq 5$ in $W_0\bdot W_2$ are $e_2$ or $e_3$. Assuming w.l.o.g. that $e'_1=e_2$ is the term with multiplicity  $n$, we find that $e_3$ and $e_2+e_3$ must be the two terms of $W_0\bdot W_2$ with multiplicity $n-1$.
In such case, we have w.l.o.g.  $e'_2=e_3$ and $ye'_1+e'_2=e'_3=e_2+e_3=e'_1+e'_2$, contradicting that $y\in [2,n-2]$.
Therefore we must have $\vp_{e_2+e_3}(W_0\bdot W_2)<n-1$, ensuring that $$e_2+e_3=e'_2+e'_3.$$

Since $e'_2+e'_3=e_2+e_2$  has multiplicity $1$ in $W_0\bdot W_2$ and occurs in $W_0$, we cannot have $e_2+e_3=g\in \supp(W_2)$, meaning  $$g\notin \{e_1,e_2,e_3,e_2+e_3\}$$ with $\vp_g(W_2)=\vp_g(W_0\bdot W_2)\in \{n-1,n\}$ (as all other terms in $\supp(W_0\bdot W_2)=\{e'_1,e'_2,e'_3,e'_2+e'_3\}$ have multiplicity $n-1$ or $n$). However, since the $n$-term zero-sum $W_2$ cannot contain a term with multiplicity exactly $n$, it follows that $\vp_g(W_2)=n$, whence $W_2=g^{[n]}$. Thus $e'_1=g$, and w.l.o.g. $e'_2=e_2$ and $e'_3=e_3$, so
$$g=e'_1= y^*(e'_2-e'_3)= y^*(e_2-e_3)\in \la x^*(e_2-e_3)\ra=\la e_1\ra,$$ where $y^*\in [2,n-2]$ is the multiplicative inverse of $-y$ modulo $n$. But now \eqref{useful-obs} ensures that $g=e_1$, contradicting \eqref{gnotthere}, which completes the subcase.

\subsection*{Subcase 1.2} Lemma \ref{lemma-casen}.1  holds for $W_0\bdot W_2$, say with basis $(f_1,f_2)$.

If $f_1\notin \{e_2,e_3,e_2+e_3\}$,
then, since $e_2,e_3, e_2+e_3\in \supp(W_0)$, Lemma \ref{lemma-casen}.1 implies $e_3=(e_2+e_3)-e_2\in \la f_1\ra$ and $e_2=(e_2+e_3)-e_3\in \la f_1\ra$, implying $G=\la e_2,e_3\ra=\la f_1\ra$, which is not possible. Therefore  we must have $f_1\in \{e_2,e_3,e_2+e_3\}$. If $f_1=e_2+e_3$, then Lemma \ref{lemma-casen}.1 implies $e_2-e_3\in \la f_1\ra=\la e_2+e_3\ra$, which is not possible in view of $n\geq 3$. Thus $f\in \{e_2,e_3\}$, and by the symmetry between $e_2$ and $e_3$, we can w.l.o.g. assume  $f_1=e_2$.
In consequence, Lemma \ref{lemma-casen}.1 holding  for $W_0\bdot W_2$ implies \be\label{supp-tactic}\supp(W_2)\subseteq \{e_2\}\cup \big(\la e_2\ra+e_3\big).\ee Since there is some $g\in \supp(W_2)$ with $g\notin \{e_1,e_2,e_3\}$, it follows in view of \eqref{supp-tactic} and \eqref{useful-obs} that $$g=(y-1)e_2+e_3=xe_1+ye_2\quad\mbox{ for some $y\in [2,n-1]$}.$$

Suppose $y=2$, meaning $g=e_2+e_3=xe_1+2e_2\in \supp(W_2)$. In view of $\vp_{e_2+e_3}(W_0)=1$ and \eqref{S'-define},  $$R=e_2^{[n-1]}\bdot (xe_1+e_2)^{[n-1]}\bdot (xe_1+2e_2)^{[2]}\bdot e_1^{[n-1]}$$ is a subsequence of $S$. Observe that \be\label{listout}\Sigma_{n-4}(e_2^{[n-1]}\bdot (xe_1+e_2)^{[n-1]})=(n-4)e_2+\{0, xe_1,2xe_1,\ldots, (n-4)xe_1\}.\ee If $(-2x-1)e_1-4e_2\in \Sigma_{n-4}(e_2^{[n-1]}\bdot (xe_1+e_2)^{[n-1]})$, say with $T\mid e_2^{[n-1]}\bdot (xe_1+e_2)^{[n-1]}$ having $|T|=n-4$ and $\sigma(T)=(-2x-1)e_1-4e_2$, then $e_1\bdot (xe_1+2e_2)^{[2]}\bdot T$ will be a zero-sum subsequence of $R$ with length $n-1$, contradicting the hypothesis that $0\notin \Sigma_{\leq n-1}(S)$. Therefore  we must have $(-2x-1)e_1-4e_2\notin \Sigma_{n-4}(e_2^{[n-1]}\bdot (xe_1+e_2)^{[n-1]})$, which combined with \eqref{listout} forces $(-2x-1)e_1\in \{(n-3)xe_1, (n-2)xe_1, (n-1)xe_1\}$ since $\gcd(x,n)=1$. Hence either $x=1$ or $x=n-1$, both contradicting that $x\in [2,n-2]$. So we instead conclude that $$g=xe_1+ye_2\in \supp(W_2) \quad\mbox{ with} \quad y\in [3,n-1].$$

We have  $$R=e_2^{[n-1]}\bdot (xe_1+e_2)^{[n-1]}\bdot (xe_1+ye_2)\bdot e_1^{[n-1]}$$ as a subsequence of $S$. Observe that \be\label{listout-2}(xe_1+ye_2)+\Sigma_{n-y}(e_2^{[n-1]}\bdot (xe_1+e_2)^{[n-1]})=\{xe_1, 2xe_1,\ldots, (n-y+1)xe_1\}.\ee
If $(n-z)e_1\in (xe_1+ye_2)+\Sigma_{n-y}(e_2^{[n-1]}\bdot (xe_1+e_2)^{[n-1]})$ for some $z\in [0,y-2]$, say with $T\mid e_2^{[n-1]}\bdot (xe_1+e_2)^{[n-1]}$ a subsequence with $|T|=n-y$ and $\sigma((xe_1+ye_2)\bdot T)=(n-z)e_1$, then $(xe_1+ye_2)\bdot T\bdot e_1^{[z]}$ will be a zero-sum subsequence of $R$ with length $n-y+1+z\leq n-1$, contradicting the hypothesis that $0\notin \Sigma_{\leq n-1}(S)$. Therefore we instead conclude that $$(xe_1+ye_2)+\Sigma_{n-y}(e_2^{[n-1]}\bdot (xe_1+e_2)^{[n-1]})\subseteq \{e_1,2e_1,\ldots, (n-y+1)e_1\},$$ which is an arithmetic progression with difference $e_1$ and length $n-y+1$. In view of \eqref{listout-2}, we see that $(xe_1+ye_2)+\Sigma_{n-y}(e_2^{[n-1]}\bdot (xe_1+e_2)^{[n-1]})$ has size $n-y+1\in [2,n-2]$ and equals an arithmetic progression with difference $xe_1$. As a result, since the difference $d$ of an arithmetic progression of length $\ell\in [2,\ord(d)-2]$ is unique up to sign, we conclude that $xe_1=\pm e_1$, contradicting that $x\in [2,n-2]$, which completes Case 1.

\subsection*{Case 2:} Every zero-sum subsequence $S'\mid S$ with $|S|=3n-1$  satisfies Lemma \ref{lemma-casen}.1.

Let $S=W_0\bdot W_1\bdot\ldots\bdot W_s$ be an arbitrary block decomposition. Apply the induction hypothesis to $W_0\bdot W_1\bdot\ldots\bdot W_{s-1}$. If Item 2, holds, then there will be a subsequence $S'\mid W_0\bdot W_1\bdot\ldots\bdot W_{s-1}$ with $|S'|=3n-1$ satisfying Lemma \ref{lemma-casen}.2, contrary to case hypothesis. Therefore we can assume Item 1 holds, say with basis $(e_1,e_2)$, and by choosing the representative $e_2\in \la e_1\ra+e_2$ appropriately, we can w.l.o.g. assume $e_2\in \supp(W_0)$.
We then have $$\supp(W_0\bdot W_1\bdot \ldots\bdot W_{s-1})\subseteq \{e_1\}\cup \big(\la e_1\ra+e_2\big)$$ with $\vp_{e_1}(W_0\bdot W_1\bdot \ldots\bdot W_{s-1})\equiv -1\mod n$.
As noted in Section \ref{sec-intro}, any $n$-term zero-sum sequence $T$ with $\supp(T)\subseteq\{ e_1\}\cup \big(\la e_1\ra+e_2\big)$ must have $T=e_1^{[n]}$ or $T=\prod^\bullet_{i\in [1,n]}(y_ie_1+e_2)$ for some $y_1,\ldots,y_n\in [0,n-1]$ with $y_1+\ldots+y_n\equiv 0\mod n$. Since the multiplicity of $e_1$ is congruent to $0$ modulo $n$ in any such $T$, it follows that $\vp_{e_1}(W_0)\equiv n-1\mod n$, forcing $\vp_{e_1}(W_0)=n-1$ as $|W_0|=2n-1$.
Thus $$W_0=e_1^{[n-1]}\bdot{\prod}^{\bullet}_{i\in [1,n]}(x_ie_1+e_2)$$ for some $x_1,\ldots,x_n\in [0,n-1]$ with $x_1+\ldots+x_n\equiv 1\mod n$ (and some $x_i=0$ since $e_2\in \supp(W_0)$).
Apply Lemma \ref{lemma-casen} to $W_0\bdot W_s$. By case hypothesis, Lemma \ref{lemma-casen}.1 must hold for $W_0\bdot W_s$, say with basis $(f_1,f_2)$.

If $f_1=e_1$, then $e_2\in \supp(W_0)$ and Lemma \ref{lemma-casen}.1 imply  that $\supp(W_0\bdot W_s)\subseteq \{e_1\}\cup \big(\la e_1\ra+e_2\big)$, and Item 1 follows with basis $(e_1,e_2)$. If $f_1\notin \supp(W_0)$, then Lemma \ref{lemma-casen}.1 implies $e_1-e_2\in \la f_1\ra$ and $x_ie_1=(x_ie_1+e_2)-e_2\in \la f_1\ra$ for all $i=1,\ldots,n$.
Thus $\la x_1e_1,\ldots,x_ne_1\ra\subseteq \la f_1\ra$. Since $ x_1+\ldots+x_n\equiv 1\mod n$, it follows that $\gcd(x_1,\ldots,x_n,n)=1$, whence $e_1\in \la e_1\ra=\la x_1e_1,\ldots,x_ne_1\ra\subseteq \la f_1\ra$.
Combined with the fact that $e_1-e_2\in \la f_1\ra$, we deduce that $G=\la e_1-e_2, e_1\ra\subseteq \la f_1\ra$, which is not possible. Therefore we instead conclude that $f_1\in \supp(W_0)$, and since $f_1\neq e_1$, we must have $f_1=ze_1+e_2$ for some $z\in [0,n-1]$. Replacing $e_2$ by $f_1\in \supp(W_0)$, we can assume $f_1=e_2$. But now $\supp(W_0)\subseteq \Big(\{e_1\}\cup \big(\la e_1\ra+e_2\big)\Big)\cap \Big(\{e_2\}\cup \big(\la e_2\ra+e_1\big)\Big)=\{e_1,e_2,e_1+e_2\}$. Since $x_1+\ldots+x_n\equiv 1\mod n$, this is only possible if $W_0=e_1^{[n-1]}\bdot e_2^{[n-1]}\bdot (e_1+e_2)$.

As we began with an arbitrary block decomposition, we see that proof is complete unless every block decomposition $S=W_0\bdot W_1\bdot\ldots \bdot W_s$ has $$W_0=e_1^{[n-1]}\bdot e_2^{[n-1]}\bdot (e_1+e_2)$$ for some basis $(e_1,e_2)$ for $G$ with $\supp(S\bdot W_s^{[-1]})\subseteq\{ e_1\}\cup \big(\la e_1\ra+e_2\big)$ and with Lemma \ref{lemma-casen}.1 holding for $W_0\bdot W_s$ using the basis $(e_2,e_1)$. If we can show there is a block decomposition that does not satisfy these conditions, then the proof will be complete.

We have  $\supp(W_s)\subseteq \{e_2\}\cup \big(\la e_2\ra+e_1\big)$ since Lemma \ref{lemma-casen}.1 holds for $W_0\bdot W_s$ using the basis $(e_2,e_1)$. Thus either $W_s=e_2^{[n]}$, in which case Lemma \ref{lemma-casen-gen}.1 holds for $S$ using the basis $(e_1,e_2)$, or else $$\supp(W_s)\subseteq \la e_2\ra+e_1.$$ As we are done in the former case, we assume the latter holds. But now there must be some $$ye_2+e_1\in \supp(W_s)\quad\mbox{ with $y\in [2,n-1]$},$$ for if $\supp(W_s)\subseteq \{e_1,e_2+e_1\}$, then $W_s=e_1^{[n]}$ or $W_s=(e_2+e_1)^{[n]}$ for the $n$-term zero-sum $W_s$, in which case Lemma \ref{lemma-casen-gen}.1 holds for $S$ using the basis $(e_1,e_2)$. Likewise, we must also have some $xe_1+e_2\in \supp(W_1\bdot\ldots\bdot W_{s-1})$ with $x\in [2,n-1]$, else Lemma \ref{lemma-casen-gen}.1 holds for $S$ using the basis $(e_2,e_1)$. We may w.l.o.g. assume $$xe_1+e_2\in \supp(W_1)\subseteq  \la e_1\ra+e_2\quad\mbox{ with $x\in [2,n-1]$}.$$

Suppose $\vp_{e_2}(W_1)\geq 1$. Then $\vp_{xe_1+e_2}(W_0\bdot W_1)=\vp_{xe_1+e_2}(W_1)\leq n-2$ (as the $n$-term zero-sum  $W_1$ cannot have a term with multiplicity exactly $n-1$). Set $W'_1=e_2^{[n]}$ and $W'_0=W_0\bdot W_1\bdot e_2^{[-n]}$.
Then $S=W'_0\bdot W'_1\bdot W_2\bdot\ldots\bdot W_s$ is a block decomposition with $e_1+e_2,xe_1+e_2\in \supp(W'_0)$, $\vp_{e_1}(W'_0)=n-1$ and
$\vp_{xe_1+e_2}(W'_0)\leq n-2$. As a result, the new block decomposition can only have the required form if  $e_1+(e_1+e_2)=xe_1+e_2$ with $\vp_{e_1+e_2}(W'_0)=n-1$ and Lemma \ref{lemma-casen}.1 holding for $W'_0\bdot W_s$ using the basis $(e_1+e_2,e_1)$ or $(e_1,e_1+e_2)$. Hence $x=2$,
 and either  $\supp(W_s)\subseteq \Big(\la e_2\ra+e_1\Big)\cap \Big(\{e_1+e_2\}\cup \big(\la e_1+e_2\ra+e_1\big)\Big)=\{e_1+e_2,e_1\}$ or $\supp(W_s)\subseteq \Big(\la e_2\ra+e_1\Big)\cap \Big(\{e_1\}\cup \big(\la e_1\ra+e_1+e_2\big)\Big)=\{e_1+e_2,e_1\}$, both contradicting that $e_1+ye_2\in \supp(W_s)$ with $y\in [2,n-1]$. So we conclude that $e_2\notin \supp(W_1)$.

Now $R=W_0\bdot W_1\bdot (xe_1+e_2)^{[-1]}$ has $|R|=3n-2=\mathsf s_{\leq n}(G)$, and thus there must be a zero-sum subsequence $W'_0\mid R$ with $|W'_0|=n$. As a result, there is a block decomposition $S=W'_0\bdot W'_1\bdot W_2\bdot\ldots \bdot W_s$ with $$xe_1+e_2\in \supp(W'_0).$$
Since $\vp_{e_1}(W_0\bdot W_1)=\vp_{e_2}(W_0\bdot W_1)=n-1$ (in view of $e_2\notin \supp(W_1)$ and $\supp(W_1)\subseteq  \la e_1\ra+e_2$), it follows that $$e_1,e_2\in \supp(W'_0),$$ since an $n$-term zero-sum cannot contain a term with multiplicity exactly $n-1$. Since $e_1+e_2\neq xe_1+e_2$ (as $x\in [2,n-1]$) and $e_2+(xe_1+e_2)\neq e_1$, it follows that the new block decomposition can only have the required form if $e_1+(xe_1+e_2)=e_2$ with $\vp_{e_1}(W'_0)=\vp_{xe_1+e_2}(W'_0)=n-1$ and Lemma \ref{lemma-casen}.1 holding for $W'_0\bdot W_s$ using the basis $(xe_1+e_2,e_1)$ or $(e_1,xe_1+e_2)$. Hence $x=n-1$ and either $\supp(W_s)\subseteq \big(\la e_2\ra+e_1\big)\cap \Big(\{-e_1+e_2\}\cup \big(\la -e_1+e_2\ra+e_1\big)\Big)=\{e_1\}$ or $\supp(W_s)\subseteq \big(\la e_2\ra+e_1\big)\cap \Big(\{e_1\}\cup \big(\la e_1\ra-e_1+e_2\big)\Big)=\{e_1,e_1+e_2\}$, again contradicting that $e_1+ye_2\in \supp(W_s)$ with $y\in [2,n-1]$, which completes the proof.
\end{proof}

The following proposition now handles the missing case in the proof of the main result from \cite{propB-GGG}. Lemma \ref{lemma-casen-gen} replaces the use of \cite[Proposition 4.2]{propB-GGG}. After invoking Lemma \ref{lemma-casen-gen}, if ever  Lemma \ref{lemma-casen-gen}.2 holds, then applying Proposition \ref{PropBFix} immediately concludes the proof, allowing one to assume Lemma \ref{lemma-casen-gen}.1 always holds, which was simply assumed to be the case in \cite{propB-GGG} due to the faulty statement of
\cite[Proposition 4.2]{propB-GGG}. Of course, once the main result of \cite{propB-GGG} is proved, it can be used to show \cite[Proposition 4.2]{propB-GGG} is retrospectively correct.

\begin{proposition}
\label{PropBFix}
Let $m\geq 4$ and $n\geq 2$  and let $G=(\Z/mn\Z)^2$. Suppose $S\in\mathcal F(G)$ is a minimal zero-sum sequence of length $|S|=2mn-1$ and that Property B holds for both $
(\Z/m\Z)^2$ and $(\Z/n\Z)^2$.  Let $\varphi: G\rightarrow G$ be the multiplication by $m$ homomorphism.
\begin{itemize}
\item[1.] $\varphi(S)$ is zero-sum with $0\notin \Sigma_{\leq n-1}(\varphi(S))$.
\item[2.] If Lemma \ref{lemma-casen-gen}.2 holds for $\varphi(S)$, then Property A holds for $S$.
\end{itemize}
\end{proposition}

\begin{proof}
Observe that $$\ker\varphi =nG\cong (\Z/m\Z)^2\quad\und\quad\varphi(G)=mG\cong (\Z/n\Z)^2.$$
1. Suppose there were some nontrivial $W_1\mid S$ with $\varphi(W_1)$ zero-sum and $|W_1|\leq n-1$. Then, since $|S|=2mn-1=(n-1+(2m-4)n)+3n$, repeated application of the definition of $\mathsf s_{\leq n}(mG)=\mathsf s_{\leq n}((\Z/n\Z)^2)=3n-2$ yields disjoint subsequences $W_1\bdot \ldots \bdot W_{2m-2}\mid S$ with $|W_1|\leq n-1$, $|W_i|\leq n$ for all $i\in [2,2m-2]$, and $\varphi(W_i)$ zero-sum for all $i\in[1,2m-2]$. In such case, $W=S\bdot (W_1\bdot\ldots \bdot W_{2m-2})^{[-1]}$ is a subsequence with $|W|\geq 2n$ and $\varphi(W)$ zero-sum. Thus, since $\mathsf D(mG)=\mathsf D((\Z/n\Z)^2)=2n-1$, we have a factorization $W=W_{2m-1}\bdot W_{2m}$
with $\varphi(W_{2m-1})$ and $\varphi(W_{2m})$ both nontrivial zero-sum subsequences.
But now $\prod^\bullet_{i\in [1,2m]}\sigma(W_i)\in\Fc(\ker \varphi)$ is a zero-sum sequence of length $2m>2m-1=\mathsf D((\Z/m\Z)^2)=\mathsf D(\ker \varphi)$, ensuring there is some proper, nontrivial zero-sum subsequence, say $\prod^\bullet_{i\in I}\sigma(W_i)$. In such case, the sequence $\prod_{i\in I}^\bullet W_i$ is a proper, nontrivial zero-sum subsequence of $S$, contradicting that $S$ is a minimal zero-sum sequence.

2.
Since Lemma \ref{lemma-casen-gen}.2 holds for $\varphi(S)$, it follows that   $n\geq 5$ and  there is a basis $(\overline e_1,\overline e_2)$ for $mG$  such that \be\label{S-struct}\varphi(S)=\overline e_1^{[an]}\bdot \overline e_2^{[bn-1]}\bdot (x\overline e_1+\overline e_2)^{[cn-1]} \bdot (x\overline e_1+2\overline e_2)\ee for some $x\in [2,n-2]$ with $\gcd(x,n)=1$, and some $a,b,c\geq 1$ with $a+b+c=2m$. Set \be\label{e2e3-rel-e1}\overline e_3=x\overline e_1+\overline e_2, \quad\mbox{ so } \quad \overline e_2=(n-x)\overline e_1+\overline e_3,\ee and note that $\overline e_1=x^*(\overline e_2-\overline e_3)$, where $x^*\in [2,n-2]$ is the multiplicative inverse of $-x$ modulo $n$, so $$x^*x\equiv -1\mod n \quad\mbox{ with $x^*\in [2,n-2]$}.$$
We remark that there is complete symmetry between $\overline e_2$ and $\overline e_3$, for if we swap the role of $\overline e_2$ and $\overline e_3$, all the above setup remains true replacing the arbitrary  $x,\,x^*\in [2,n-2]$ by $n-x,\,n-x^*\in [2,n-2]$.
For each $\overline e_i$, we can choose some $e_i\in G$ such that $$\varphi (e_i)=me_i=\overline e_i\quad\mbox { for $i\in [1,3]$}.$$ For the moment, let this choice of representatives for the $\overline e_i$ be arbitrary except that $$e_3=xe_1+e_2.$$ We will later specialize the values of $e_1$, $e_2$ and (consequently) $e_3$ by exchanging $e_i$ for a more specific element from $e_i+\ker\varphi=e_i+nG$.

 In view of \eqref{S-struct}, we have a decomposition $S=W_0\bdot W_1\bdot W_2\bdot\ldots\bdot W_{2m-2}$ with \begin{align*} \varphi(W_0)=\overline e_1^{[n-1]}\bdot \overline e_2^{[x^*]}\bdot \overline e_3^{[n-x^*]},\quad
&\varphi(W_1)=\overline e_3^{[x^*-1]}\bdot \overline e_2^{[n-x^*-1]}\bdot \overline e_1\bdot (\overline e_2+\overline e_3),\quad \und\\
&\varphi(W_i)\in \{\overline e_1^{[n]},\,\overline e_2^{[n]},\, \overline e_3^{[n]}\} \quad\mbox {for $i\in [2,2m-2]$}.\end{align*} We call any such decomposition of $S$ a \emph{block decomposition} with each $W_i$ a block. Note $\varphi(W_i)$ is zero-sum for all $i\in [0,2m-2]$ when $S=W_0\bdot\ldots\bdot W_{2m-2}$ is a block decomposition in view of the definition of $\overline e_3$ and $x^*$. A decomposition $S=W_0\bdot\ldots\bdot W_{2m-2}$ with each $\varphi(W_i)$ for $i\in [0,2m-2]$ a nontrivial zero-sum subsequence will be called a \emph{weak block decomposition} of $S$. For any such weak block decomposition, the associated sequence $S_\sigma=\sigma(W_0)\bdot\ldots\bdot \sigma(W_{2m-2})\in \Fc(nG)$ must be a minimal zero-sum sequence of length $2m-1$.
Indeed, since $|S_\sigma|=2m-1=\mathsf D((\Z/m\Z)^2)=\mathsf D(nG)$, it contains a nontrivial zero-sum subsequence, say $\prod_{i\in I}^\bullet \sigma(W_{i})$, which cannot be proper lest the subsequence $\prod_{i\in I}^\bullet W_{i}$ contradict that $S$ is a minimal zero-sum sequence.

For the moment, consider an arbitrary block decomposition $S=W_0\bdot\ldots\bdot W_{2m-2}$.
Since Property B holds for $(\Z/m\Z)^2$ and $S_\sigma$ is a minimal zero-sum of length $2m-1$, let  $(f_1,f_2)$ be a basis for $nG\cong (\Z/m\Z)^2$ such that Property A holds for $S_\sigma$.  If possible, choose the block decomposition so that $S_{\sigma}$ has $f_1$ as the  unique term with multiplicity $m-1$. If this is not possible, then choose $(f_1,f_2)$ such that $S_\sigma=f_1^{[m-1]}\bdot f_2^{[m-1]}\bdot (f_1+f_2)$.
Now let $S^*=S_\sigma$ be a fixed sequence for which these conditions hold (so either $f_1$ is the unique term with multiplicity $m-1$ in $S^*$ or $S^*=f_1^{[m-1]}\bdot f_2^{[m-1]}\bdot (f_1+f_2)$). During the proof, we will range over various block decompositions of $S$, but all having the same associated sequence $S_\sigma=S^*$. Likewise, the basis $(f_1,f_2)$ for which Property A holds for $S^*$ will  remain fixed during the proof, except when $S^*=f_1^{[m-1]}\bdot f_2^{[m-1]}\bdot (f_1+f_2)$ when we will apply our arguments using both the basis $(f_1,f_2)$ and the basis $(f_2,f_1)$.

For $g\in \supp(S)$, we have $\varphi(g)\in \{\overline e_1,\overline e_2,\overline e_3, \overline e_2+\overline e_3\}$. If $\varphi(g)=\overline e_i$, then $g=e_i+\psi(g)$ for some $\psi(g)\in \ker \varphi=nG$. If $\varphi(g)=\overline e_2+\overline e_3$, then $g=e_2+e_3+\psi(g)=xe_1+2e_2+\psi(g)$ for some $\psi(g)\in \ker \varphi=nG$. Since $(f_1,f_2)$ is a basis for $nG$, we have each $\psi(g)=yf_1+zf_2$, and set $\psi_1(g)=yf_1$ and $\psi_2(g)=zf_2$. In this way, we define the functions $\psi$, $\psi_1$ and $\psi_2$, which depend upon the choice of representatives $e_1$, $e_2$ and $e_3=xe_1+e_2$, as well as the basis $(f_1,f_2)$ in case $S^*=f_1^{[m-1]}\bdot f_2^{[m-1]}\bdot (f_1+f_2)$.
Additionally, for $k\in [1,3]$ and $i\in [0,2m-2]$, let
\begin{align*}&W_i^{(k)}\mid W_i\mbox{ consist of all terms $g$ with $\varphi(g)=\overline e_k$}\quad\und\\
& S_k\mid S \mbox{ consist of all terms $g$ with $\varphi(g)=\overline e_k$}.
\end{align*}

Every block $W_i$ has  $\sigma(W_i)=f_1$ or $\sigma(W_i)=y_if_1+f_2$ for some $y_i\in [0,m-1]$. We call the former type (I) blocks, and the latter type (II) blocks. Partition $[0,2m-2]=B_{(I)}\cup B_{(II)}$ with $B_{(I)}$ indexing the type (I) blocks and $B_{(II)}$ indexing the type (II) blocks, in which case
$$S^*=S_\sigma=f_1^{[m-1]}\bdot {\prod}^\bullet_{i\in B_{(II)}}(y_if_1+f_2)$$  with $|B_{(II)}|=m$ and  $\Summ{i\in B_{(II)}}y_i\equiv 1\mod m$. In particular, $\sigma(W_i)=y_if_1+f_2$ cannot be constant on all $m$ type (II) blocks $i\in B_{(II)}$, as this would contradict that $\Summ{i\in B_{(II)}}y_i\equiv 1\not\equiv 0\mod m$.
For $j\in B_{(II)}$, set $$F_j=(y_j-1)f_1+f_2.$$
Note, in the case $S_\sigma=f_1^{[m-1]}\bdot f_2^{[m-1]}\bdot (f_1+f_2)$, that exchanging the basis $(f_1,f_2)$ for the basis $(f_2,f_1)$ has the effect  of making the blocks $W_i$ with $\sigma(W_i)=f_1$ into type (II) blocks and the blocks $W_i$ with $\sigma(W_i)=f_2$ into type (I) blocks.

If $W_j$ and $W_k$ are distinct blocks, meaning $j,\,k\in [0,2m-2]$ are distinct, with $T\mid W_j$ and $R\mid W_k$ nontrivial sequences with $\sigma(\varphi(T))=\sigma(\varphi(R))$, then setting $W'_j=W_j\bdot T^{[-1]}\bdot R$, $W'_k=W_k\bdot R^{[-1]}\bdot T$ and $W'_i=W_i$ for $i\neq j,k$ results in a new weak block decomposition $S=W'_0\bdot \ldots\bdot W'_{2m-2}$ with associated sequence $S'_\sigma=\sigma(W'_0)\bdot\ldots\bdot \sigma(W'_{2m-2})$ obtained by \emph{swapping} $T\mid W_j$ with $R\mid W_k$. If we wish to be more specific, we call this a \emph{$(T,R)$-swap}.
The special case when  $W_j=W_0$ with $\varphi(T)=\overline e_1^{[x]}\bdot \overline e_2$ and $\varphi(R)=\overline e_3$ is abbreviated as an $(\overline e_2+,\overline e_3)$-swap, and the case when $W_j=W_0$ with  $\varphi(T)=\overline e_1^{[n-x]}\bdot \overline e_3$ and $\varphi(R)=\overline e_2$ is abbreviated as an $(\overline e_3+,\overline e_2)$--swap.
 Letting $S'_\sigma$ be the associated sequence of the new weak block decomposition, we can apply Lemmas \ref{lemma-pertub-I}, \ref{lemma-pertub-II(weak)} or \ref{lemma-pertub-III(strong)} to deduce restrictions on the value $\sigma(\psi(T))-\sigma(\psi(R))$,  which is possible in view of the hypothesis $m\geq 4$.
 If there is a unique term with multiplicity $m-1$ in $S^*$, then we use Lemma \ref{lemma-pertub-I}. If $S^*=f_1^{[m-1]}\bdot f_2^{[m-1]}\bdot (f_1+f_2)$, then (in general) we must instead use the weaker Lemma \ref{lemma-pertub-II(weak)} unless  $\varphi(T)=\varphi(R)$, in which case we can leverage the fact that every block decomposition has two terms with multiplicity $m-1$ (per our assumptions)  to  gain access to the stronger Lemma \ref{lemma-pertub-III(strong)}.

 \subsection*{Claim A} Let $k\in [1,3]$ and let $I_k\subseteq [0,2m-2]$ be all indices $i$ with $\overline e_k\in \supp(\varphi(W_i))$.

 \begin{itemize}

 \item[1.]If every $i\in I_k$ has $W_i$ of type (I), then $\psi$ is constant on $S_k$.
 \item[2.] If every $i\in I_k$ has $W_i$ of type (II), then $\psi_2$ is constant on $S_k$.
 \item[3.] In all other cases, either $\psi$ is constant on $S_k$ or one of the following scenarios holds.
     \begin{itemize}
     \item[(a)]There is a unique $i_0\in I_k$ with $W_{i_0}$ having type (I), there is some $y\in [0,m-1]$ such that $\sigma(W_i)=yf_1+f_2$ for all $i\in I_k\setminus\{i_0\}$, and there is some $g_0\in \supp(W^{(k)}_{i_0})$  such that $\psi$ is constant on $S_k\bdot g_0^{[-1]}$, say $\psi(g)=z$ for all $g\in \supp(S_k\bdot g_0^{[-1]})$, with $\psi(g_0)=z-F$, where $F=(y-1)f_1+f_2$.
     \item[(b)] There is a unique $j_0\in I_k$ with $W_{j_0}$ having type (II) and some $g_0\in \supp(W^{(k)}_{j_0})$  such that $\psi$ is constant on $S_k\bdot g_0^{[-1]}$, say $\psi(g)=x$ for all $x\in \supp(S_k\bdot g_0^{[-1]})$, with $\psi(g_0)=x+F_{j_0}=x+(y_{j_0}-1)f_1+f_2$.
     \end{itemize}
 \end{itemize}

\begin{proof}
First observe that $0,1\in I_k$ for all $k\in [1,3]$, so $|I_k|\geq 2$.
If $i,j\in I_k$ are distinct with $W_i$ and $W_j$ both of type (I),  and $g\in \supp(W^{(k)}_i)$ and $h\in \supp(W^{(k)}_j)$, then performing a $(\overline e_k,\overline e_k)$-swap using the terms $g$ and $h$ results in a new block decomposition with associated sequence $S'_\sigma$. We can apply Lemma \ref{lemma-pertub-I} or Lemma \ref{lemma-pertub-III(strong)} to conclude $S'_\sigma=S_\sigma$ and $\psi(g)=\psi(h)$. Thus, performing all possible $(\overline e_k,\overline e_k)$-swaps between $W_i$ and $W_j$, we find that $\psi$ is constant on all terms from $S_k$ contained in these blocks. If all blocks $W_i$ with $i\in I_k$ have type (I), applying this argument to all possible pairs shows that $\psi$ is constant on $S_k$, which completes Item 1. If $i,j\in I_k$ are distinct with $W_i$ and $W_j$ both of type (II),  and $g\in \supp(W^{(k)}_i)$ and $h\in \supp(W^{(k)}_j)$, then performing a $(\overline e_k,\overline e_k)$-swap using the terms $g$ and $h$ results in a new block decomposition with associated sequence $S'_\sigma$. We can apply Lemma \ref{lemma-pertub-I} or Lemma \ref{lemma-pertub-III(strong)} to  conclude  $\psi_2(g)=\psi_2(h)$. If all blocks $W_i$ with $i\in I_k$ have type (II), applying this argument to all possible pairs shows that $\psi_2$ is constant on $S_k$, which completes Item 2. It remains to prove Item 3.

Let $I_k=I^{(I)}_k\cup I^{(II)}_k$ with $I_k^{(I)}$ indexing all blocks $W_i$ having type (I), and $I_k^{(II)}$ indexing all blocks $W_i$ having type (II). By hypothesis of Item 3, both  $I_k^{(I)}$ and $I_k^{(II)}$ are nonempty. If $|I_k^{(I)}|\geq 2$, then the argument of the previous paragraph shows $\psi$ is constant on all terms from $S_k$ contained in the blocks $W_i$ with $i\in I_k^{(I)}$.  If $|I_k^{(II)}|\geq 2$, then the argument of the previous paragraph shows $\psi_2$ is constant on all terms from $S_k$ contained in the blocks $W_i$ with $i\in I_k^{(II)}$.

The following arguments from here to the beginning of the Case A.1 will be repeatedly used in what follows  and referenced as the \textbf{initial arguments}. Consider an arbitrary $i\in I_k^{(I)}$ and $j\in I_k^{(II)}$. Then $\sigma(W_i)=f_1$ and  $\sigma(W_j)=y_jf_1+f_2$ with $F_j=(y_j-1)f_1+f_2$.
Let $g\in \supp(W^{(k)}_i)$ and $h\in \supp(W^{(k)}_j)$. Performing a $(\overline e_k,\overline e_k)$-swap using the terms $g$ and $h$ results in a new block decomposition with associated sequence $S'_\sigma$. We can apply Lemma \ref{lemma-pertub-I} or Lemma \ref{lemma-pertub-III(strong)} to conclude $S'_\sigma=S_\sigma$ and either $\psi(h)=\psi(g)$ or $\psi(h)=\psi(g)+F_j$. Fixing $g$, say  with $\varphi(g)=x$, and ranging over all possible $h\in \supp(W^{(k)}_j)$, we conclude that $\supp(\psi(W^{(k)}_j))\subseteq  \{x,x+F_j\}$. Likewise, fixing  $h$, say  with $\varphi(h)=z$, and ranging over all possible $g\in \supp(W^{(k)}_i)$, we conclude that $\supp(\psi(W^{(k)}_i))\subseteq  \{z,z-F_j\}$. In particular, $|\supp(\psi(W_i^{(k)}))|\leq 2$ and $|\supp(\psi(W_j^{(k)}))|\leq 2$.

If equality holds in one of these estimates, say  $|\supp(\psi(W_j^{(k)}))|=2$, then $\supp(\psi(W_j^{(k)}))=\{x,x+F_j\}$ with $x\in \supp(\psi(W_i^{(k)}))$. Choosing $h\in \supp(W^{(k)}_j)$ with  $\psi(h)=x+F_j$ and then with $\psi(h)=x$, and applying the argument of the previous paragraph in both cases, we find that $\supp(\psi(W^{(k)}_i))\subseteq \{\psi(h),\psi(h)-F_j\}= \{x+F_j,x\}$ and  $\supp(\psi(W^{(k)}_i))\subseteq \{\psi(h),\psi(h)-F_j\}= \{x, x-F_j\}$. Thus $\supp(\psi(W_i^{(k)}))$ is contained in the intersection of these two sets, and   since $\ord(F_j)=m\geq 3$, we have $x-F_j\neq x+F_j$, meaning  $\supp(\psi(W^{(k)}_i))=\{x\}$. In this case,   $\supp(\psi(W^{(k)}_i\bdot W^{(k)}_j))\subseteq \{x, x+F_j\}$ with $\supp(\psi(W_i^{(k)}))=\{x\}$. Likewise, if $|\supp(\psi(W_i^{(k)}))|=2$, then this same argument shows
 $\supp(\psi(W^{(k)}_i\bdot W^{(k)}_j))\subseteq \{z, z-F_j\}$ with $\supp(\psi(W_j^{(k)}))=\{z\}$.
 In the remaining case when $\psi$ is constant both on $W^{(k)}_i$ and also on $W^{(k)}_j$, say with $\supp(\psi(W_i^{(k)}))=\{x\}$ and $\supp(\psi(W_j^{(k)}))=\{z\}$, then we instead find  $z\in \{x,x+F_j\}$.

If $|W_i^{(k)}|\geq 2$ and $\supp(\psi(W_i^{(k)}))=\{x\}$, then the above work shows that $\supp(\psi(W_j^{(k)}))\subseteq \{x,x+F_j\}$. If there were also at least two terms $h_1\bdot h_2\mid W_j^{(k)}$ with $\psi(h_1)=\psi(h_2)=x+F_j$, then we could perform a $(\overline e_k^{[2]},\overline e_k^{[2]})$-swap using these terms and two terms $g_1\bdot g_2\mid W_i^{(k)}$ (possible as  $|W_i^{(k)}|\geq 2$), and then apply Lemma \ref{lemma-pertub-I} or Lemma \ref{lemma-pertub-III(strong)} to the resulting associated sequence $S'_\sigma$ to find that $S_\sigma=S'_{\sigma}$ with $2F_j=\psi(h_1)+\psi(h_2)-\psi(g_1)-\psi(g_2)\in \{0, F_j\}$, contradicting that $\ord(F_j)=m\geq 3$.
Therefore we conclude that at most one term $g_0$ of $W_j^{(k)}$ can have $\psi(g_0)=x+F_j$.  If $|W_j^{(k)}|\geq 2$ and $\supp(\psi(W_j^{(k)}))=\{z\}$, then repeating this argument with the roles of $i$ and $j$ swapped, we conclude that $\supp(\psi(W_i^{(k)}))\subseteq \{z,z-F_j\}$ and that at most one term $g_0$ of $W_i^{(k)}$ can have $\psi(g_0)=z-F_j$.

\subsection*{Case A.1} $|I_k^{(I)}|\geq 2$ and $|I_k^{(II)}|\geq 2$

In this case,  our initial arguments show $\psi$ is constant on all terms of $S_k$ contained in type (I) blocks, say equal to $x$, while $\psi_2$ is constant on all terms of $S_k$ contained in type (II) blocks with $\psi(h)\in \{x,x+F_j\}$ for any $h\in \supp(W_j^{(k)})$ with $j\in I_k^{(II)}$.
Since $\psi_2(x)\neq \psi_2(x+F_j)$, we further conclude that either $\psi(h)=x$ for all  terms of $S_k$ contained in type (II) blocks, or else $\psi(h)=x+F_j$ whenever $h\in \supp(W_j^{(k)})$ with $j\in I_k^{(II)}$.  Since $|I_k^{(I)}|,\,|I_k^{(II)}|\geq 2$ and $W_1$ is the only possible block with $|W_1^{(k)}|=1$, there is at least one $i\in I_k^{(I)}$ with $|W_i^{(k)}|\geq 2$ and at least one $j\in I_k^{(II)}$ with $|W_j^{(k)}|\geq 2$. We have  $\supp(\psi(W_i^{(k)}))=\{x\}$ (as noted above), and now our initial arguments ensure that  $\psi$ cannot be constant and equal to $x+F_j$ on $W_j^{(k)}$, meaning we instead have $\psi(h)=x$ for all  terms of $S_k$ contained in type (II) blocks, and thus also on all of $S_k$, as desired.

\subsection*{Case A.2} $|I_k^{(I)}|\geq 2$ and $|I_k^{(II)}|=1$

In this case,  our initial arguments show $\psi$ is constant on all terms of $S_k$ contained in type (I) blocks, say equal to $x$, while $\psi(h)\in \{x,x+F_j\}$ for any $h\in \supp(W_{j_0}^{(k)})$, where $j_0$ is the unique element of $I_k^{(II)}$. Since $|I_k^{(I)}|\geq 2$ and $W_1$ is the only possible block with $|W_1^{(k)}|=1$, there is at least one $i\in I_k^{(I)}$ with $|W_i^{(k)}|\geq 2$, and we have $\supp(\psi(W_i^{(k)}))=\{x\}$ (as just noted). Our initial arguments now ensure that there is at most one term $g_0$ of $W_{j_0}^{(k)}$ with $\psi(g_0)=x+F_j$ while $\psi(h)=x$ for all other terms $h$ of $W_{j_0}^{(k)}$.  Thus either $\psi$ is constant on $S_k$ or Item 3(b) follows.

\subsection*{Case A.3} $|I_k^{(I)}|=1$ and $|I_k^{(II)}|\geq 2$

In this case, our initial arguments show  $\psi_2$ is constant on all terms of $S_k$ contained in type (II) blocks.  Moreover, for $j\in I_k^{(II)}$,  $\psi$ can take on at most two values over all the terms from $W_j^{(k)}$, and if two values occur, then their difference must be $\pm F_j$. However, since $\psi_2$ is constant on $W_k^{(j)}$ with $\psi_2(F_j)\neq 0$,  this latter possibility cannot occur, meaning $\psi$ is constant on each $W_j^{(k)}$ with $j\in I_k^{(II)}$, say $\supp(\psi(W_j^{(k)}))=\{z_j\}$.
Let $i_0$ be the unique index contained in $I_k^{(I)}$.

Suppose $|\supp(\psi(W_{i_0}^{(k)}))|\geq 2$. Then $\supp(\psi(W_{i_0}^{(k)}))=\{z_j,z_j-F_j\}$ for every $j\in I_k^{(II)}$ (by our initial arguments). The difference of these two elements is $F_j= (y_j-1)f_1+f_2$, and since  $\ord(f_2)=m\geq 3$, the order of which element to subtract from the other and have the resulting difference lie in $\la f_1\ra+f_2$ is determined. As this must hold for all $j\in I_k^{(II)}$, we conclude that $F_j=(y_j-1)f_1+f_2$ is constant over all $j\in I_k^{(II)}$, say with $y_j=y$ and $F_j=F=(y-1)f_1+f_2$ for all $j\in I_k^{(II)}$.
But now $\supp(\psi(W_{i_0}^{(k)}))=\{z_j,z_j-F\}$ for every $j\in I_k^{(II)}$, forcing $z_j=z$ to be constant over all $j\in I_k^{(II)}$ (as $\ord(F)=\ord(f_2)=m\geq 3$). Since $|I_k^{(II)}|\geq 2$ and $W_1$ is the only possible block with $|W_1^{(k)}|=1$, it follows that there is some $j\in I_k^{(II)}$ with $|W_j^{(k)}|\geq 2$, and $\supp(\psi(W_j^{(k)}))=\{z\}$ as already noted. Thus  our initial  arguments ensure  there is some $g_0\in \supp(W_{i_0})$ with $\psi(g_0)=z-F$ and $\psi$ constant on all of $S_k\bdot g_0^{[-1]}$, equal to $z$. In such case, Item 3(a) follows. So we can instead now assume $|\supp(\psi(W_{i_0}^{(k)}))|=1$, meaning  $\psi$ is constant on all terms from $W_j^{(k)}$ for any $j\in I_k$, including $j=i_0$.

Since $\psi$ is constant of $W_{i_0}^{(k)}$, we have    $\supp(\psi(W_{i_0}^{(k)}))=\{x\}$ for some $x$. Moreover, for each $j\in I_k^{(II)}$, we have $x\in \{z_j,z_j-F_j\}$ (by our initial arguments). If we always have $x=z_j$, then $\psi$ is constant on $S_k$, equal to $x$, as desired. Therefore we can instead assume there is some $j\in I_k^{(II)}$ such that $x=z_j-F_j$. Let $h\in \supp(W_j^{(k)})$ and $g\in \supp(W_{i_0}^{(k)})$, so $$\psi(g)=x=z_j-F_j\quad\und\quad \psi(h)=z_j.$$ We can perform a type $(\overline e_k,\overline e_k)$-swap using the terms $g$ and $h$. If we do so, the resulting new block decomposition has associated sequence $S'_\sigma=S_\sigma$ with the type of the blocks $W_{i_0}$ and $W_j$ swapping. Thus the new block $W'_{i_0}=W_{i_0}\bdot g^{[-1]}\bdot h$ has type (II) and the new block $W'_{j}=W_j\bdot h^{[-1]}\bdot g$ has type (I). Moreover,  $S'_\sigma=S_\sigma$ with the new block decomposition also satisfying the hypotheses of Case A.3. Thus by our above work applied to the new block decomposition, we conclude that $\psi$ is constant on all terms from $S_k$ in  any  block $W'_i$ with $i\in I_k$, and thus must be constant on all terms from $S_k$ contained in $W'_{i_0}=W_{i_0}\bdot g^{[-1]}\bdot h$ and also constant on all terms from $S_k$ contained in $W'_j=W_j\bdot h^{[-1]}\bdot g$. However, $h\in \supp(W'_{i_0})$ is one such term from $S_k$ with $\psi(h)=z_j$, while all other terms $g'$ from $S_k$ in $W'_{i_0}=W_{i_0}\bdot g^{[-1]}\bdot h$  have $\psi(g')=x=z_j-F_j\neq z_j$. As a result, we obtain a contradiction unless $h$ is the unique term of $W_{i_0}$ from $S_k$, which is only possible if $|W_{i_0}^{(k)}|=1$ with $i_0=1$.
Likewise, $\psi$ is constant on all terms from $S_k$ contained in $W'_j=W_j\bdot h^{[-1]}\bdot g$. However, $g\in \supp(W'_{i_0})$ is one such term from $S_k$ with $\psi(g)=x=z_j-F_j$, while all other terms $h'$ from $S_k$ in $W'_{j}=W_j\bdot h^{[-1]}\bdot g$  have $\psi(h')=z_j\neq z_j-F_j$. As a result, we obtain a contradiction unless $g$ is the unique term of $W_{j}$ from $S_k$, which is only possible if $|W_{j}^{(k)}|=1$ with $j=1$. Since we cannot have $i_0=j=1$ (as $i_0$ and $j$ are distinct), we obtain a contradiction.

\subsection*{Case A.4} $|I_k^{(I)}|=|I_k^{(II)}|=1$

We will show that $|I_k|\geq 3$ must hold for all $k\in [1,3]$, which will ensure the hypotheses of Case A.4 do not hold (else the proof is complete). If this fails,  there are $2m-3$ blocks not indexed by $I_k=\{0,1\}$, so the Pigeonhole Principle ensures that there is some $k'\in [1,3]\setminus \{k\}$ with $|I_{k'}\setminus \{0,1\}|\geq m-1$. Since $0,1\in I_{k'}$ (as this is the case for any $k'\in [1,3]$), we have $|I_{k'}|\geq m+1\geq 3$, allowing us to  apply the completed portion of Claim A to $k'$. Since there are only $m$ type (II) blocks and $|I_{k'}|\geq m+1$, not all blocks indexed by $I_{k'}$ can have type (II), meaning Item 2 cannot hold for $k'$. Since there are only $m-1$ type (I) blocks and $|I_{k'}|\geq m+1$, it follows that   $|I^{(II)}_{k'}|\geq 2$, meaning Item 3(b) cannot hold for $k'$. If  $|I^{(I)}_{k'}|=1$, then, since there are only $m$ type (II) blocks, we find that $m+1\leq |I_{k'}|\leq m+1$, which is only possible if $|I_{k'}|=m+1$ with $I_{k'}$ containing all indices of type (II) blocks.
However, since $\sigma(W_j)=y_jf_1+f_2$ is not constant on all type (II) blocks (as noted in Section \ref{sec-intro}), we see that Item 3(a) cannot hold for $k'$. The only possibility in Claim A that is not ruled out for $k'$ is that $\psi$ is constant on $S_{k'}$, in which case all terms of $S_{k'}$ are equal, meaning $S$ contains a term with multiplicity $|S_{k'}|\geq (|I_{k'}|-2)n+n-1\geq  (m-1)n+n-1=mn-1$. In such case, Property A holds for $S$ per the comments from Section \ref{sec-intro}, as desired. This shows $|I_k|\geq 3$ (else the proof is complete) and  completes Claim A.
\end{proof}

\subsection*{Claim B} $|I_k|\geq 3$ for every $k\in [1,3]$.

\begin{proof}
This follows by the argument just given in Case A.4.
\end{proof}

\subsection*{Claim C} There is a block decomposition of $S$ with $S_\sigma=S^*$ such that either $W_0$ or $W_1$ has type (I). In case $S^*=f_1^{[m-1]}\bdot f_2^{[m-1]} \bdot (f_1+f_2)$, this holds for the basis $(f_1,f_2)$ as well as the basis $(f_2,f_1)$, though not necessarily with the same block decomposition.

\begin{proof}
Assume by contradiction that  the claim is false, meaning every block decomposition has $W_0$ and $W_1$ always of type (II).
Consider an arbitrary $k\in [1,3]$ for which there is a type (I) block $W_j$   containing a term from $S_{k}$. Performing  $(\overline e_{k},\overline e_{k})$-swaps between any such type (I) block $W_j$ and either $W_0$ or $W_1$ shows via Lemma \ref{lemma-pertub-I} or Lemma \ref{lemma-pertub-III(strong)} that $\psi$ is constant on all terms from $S_{k}$ contained in $W_0\bdot W_1\bdot\prod_{i\in B_{(I)}}W_i$, else we could  perform such swaps and change the type of $W_0$ or $W_1$ from (II) to (I), contrary to assumption.
As a result, all such terms of $S$ are equal to each other. If there is only one $k\in [1,3]$ for which there is a type (I) block $W_j$ containing a term from $S_{k}$, then all $(m-1)n$ terms of $S_{k}$ contained in the type (I) blocks as well as at least $n-1$ terms of $S_{k}$ contained in $W_0\bdot W_1$ will be equal to the same element, yielding a term in $S$ with multiplicity at least $(m-1)n+n-1=mn-1$, in which case  Property A holds for $S$ as discussed in Section \ref{sec-intro},  as desired. We may therefore instead assume otherwise that the type (I) blocks contain terms from $S_k$ for at least two distinct value of $k\in [1,3]$, meaning there is some $k\in \{2,3\}$ such that the type (I) blocks contain a term from $S_k$.

Performing $(\overline e_j,\overline e_j)$-swaps between the type (II) blocks $W_0$ and $W_1$, for $j\in [1,3]$, shows via Lemma \ref{lemma-pertub-I} or Lemma \ref{lemma-pertub-III(strong)} that $\psi_2$ is constant on all terms of $S_j$ contained in $W_0\bdot W_1$.
By choosing the representatives $e_1$, $e_2$ and $e_3$, appropriately, we can thus assume $$\psi_2(g)=0\quad\mbox{ for all $g$ from $S_1\bdot S_2$ contained in $W_0\bdot W_1$}.$$  Performing type  $(\overline e_2+,\overline e_3)$-swaps between $W_0$ and $W_1$, we conclude from Lemma \ref{lemma-pertub-I} or Lemma \ref{lemma-pertub-II(weak)}  that $\psi(g)=(e_3+\psi(g))-(xe_1+e_2)\in \la f_1\ra$ for every $g$ from $W^{(3)}_1$. Thus $\psi_2(g)=0$ for all $g$ from  $W^{(3)}_1$, and since $\psi_2$ is constant on all terms from $S_3$ contained in $W_0\bdot W_1$, it follow that  $$\psi_2(g)=0\quad\mbox{ for all $g$ from $S_3$ contained in $W_0\bdot W_1$}.$$ Likewise, performing   $(\overline e_3+,\overline e_2)$-swaps
 between $W_0$ and $W_1$ and using Lemma \ref{lemma-pertub-I} or Lemma \ref{lemma-pertub-II(weak)} implies $$ne_1=((n-x)e_1+e_3)-e_2\in \la f_1\ra.$$
 As shown at the start of the claim, there is some $k\in \{2,3\}$ for which there is a type (I) block $W_j$ containing a term from $S_{k}$ with  $\psi$ constant on the terms of $W_0\bdot W_j$ contained in $S_{k}$, and thus $\psi_2$ is constant and equal to zero on all such terms (as we have $\psi_2(g)=0$ for all $g\in \supp(W_0)$). It follows that $f_1=\sigma(W_j)\in ne_{k}+\la f_1\ra$, whence either $ne_2\in \la f_1\ra$ or $ne_3\in \la f_1\ra$ (as $k\in \{2,3\}$).
 Since $e_3=xe_1+e_2$ and $ne_1\in \la f_1\ra$, it follows in either case that $$ne_1,ne_2,ne_3\in \la f_1\ra.$$ Let $W_{j'}$ be a type (II) block with $j'\in [2,2m-2]$ (possible as there are at least $m\geq 3$ such blocks) and let $k'\in [1,3]$ be the index such that $\varphi(W_{j'})=\overline e_{k'}^{[n]}$. Performing
 $(\overline e_{k'},\overline e_{k'})$-swaps between $W_0$ and $W_j$ shows via Lemma \ref{lemma-pertub-I} or Lemma \ref{lemma-pertub-III(strong)} that $\psi_2$ is constant on all terms of $W_0\bdot W_j$ from $S_{k'}$, and thus equal to $0$ since $\psi_2(g)=0$ for all $g$ from $W_0$. In such case, $y_jf_1+f_2=\sigma(W_j)\in ne_{k'}+\la f_1\ra=\la f_1\ra$, with the latter inequality in view of $ne_{k'}\in\la f_1\ra$, which is not possible, concluding Claim C.
 \end{proof}

\subsection*{Claim D} Let $k\in [1,3]$. If  $\psi$ is not constant on $S_k$, then  there is some term $g_k\in \supp(S_k)$ such that either Claim A.3(a) holds with $g_0=g_k$ for every block decomposition with $S_\sigma=S^*$ or such that Claim A.3(b) holds with $g_0=g_k$ for every block block decomposition with $S_\sigma=S^*$. In case $S^*=f_1^{[m-1]}\bdot f_2^{[m-1]} \bdot (f_1+f_2)$, this holds for the basis $(f_1,f_2)$ as well as the basis $(f_2,f_1)$, which ensures that  $\sigma(W_j)\neq f_1+f_2$ for any $j\in I_k$

\begin{proof}
Let $S=W'_0\bdot \ldots\bdot W'_{2m-2}$ be a block decomposition satisfying Claim C with associated sequence $S'_\sigma=S^*$. Let $I'_k\subseteq [0,2m-2]$ be all indices $i$ such that $\varphi(W'_i)$ contains a term equal to $\overline e_k$. Then  $0,1\in I'_k$ as this is the case for any $k\in [1,3]$, and at at least one of the blocks $W'_0$ or $W'_1$  has type (I) by Claim C. As a result, since  $\psi$ is not constant on $S_k$ yet either $W'_0$ or $W'_1$ has type (I), it follows that  Claim A.3(a) or A.3(b) must hold. In particular, there is some $g_k\in \supp(S_k)$ such that $\psi$ is constant on $S_k\bdot g_k^{[-1]}$, and since $|S_k|\geq (|I'_k|-2)n+n-1\geq 2n-1\geq 3$ by Claim B, the possibility for $g_k$ is uniquely defined irrespective of the block decomposition as the unique terms of $S_k$ whose $\psi$ value is different from the $\psi$ values of all other terms of $S_k$.
 Moreover, if Claim A.3(a) holds, then $\psi(g_k)-\psi(g)\in \la f_1\ra-f_2$ for all $g\in \supp(S_k\bdot g_k^{[-1]})$, while if Claim A.3(b) holds, then $\psi(g_k)-\psi(g)\in \la f_1\ra+f_2$ for all $g\in \supp(S_k\bdot g_k^{[-1]})$.
Now consider an arbitrary block decomposition $S=W_0\bdot\ldots \bdot W_{2m-2}$ with associated sequence $S_\sigma=S^*$. Apply Claim A to this block decomposition. Since $\psi$ is not constant on $S_k$, either Claim A.2, A.3(a) or A.3(b) holds. If  Claim A.2 holds, then $\psi(g_k)-\psi(g)\in \la f_1\ra$ for all $g\in \supp(S_k\bdot g_k^{[-1]})$, contradicting that $\psi(g_k)-\psi(g)\in \la f_1\ra\pm f_2$ for  $g\in \supp(S_k\bdot g_k^{[-1]})$. Therefore Claim A.3 holds and  there is some $h_k\in \supp(S_k)$ such that $\psi$ is constant on $S_k\bdot h_k^{[-1]}$.
Since $|S_k|\geq (|I_k|-2)n+n-1\geq 2n-1\geq 3$ by Claim B, we must have $h_k=g_k$. Moreover, if Claim A.3(a) holds for the block decomposition $S=W_0\bdot\ldots\bdot W_{2m-2}$, then $\psi(g_k)-\psi(g)\in \la f_1\ra-f_2$ for all $g\in \supp(S_k\bdot g_k^{[-1]})$, while if Claim A.3(b) holds for the block decomposition $S=W_0\bdot\ldots\bdot W_{2m-2}$, then $\psi(g_k)-\psi(g)\in \la f_1\ra+f_2$ for all $g\in \supp(S_k\bdot g_k^{[-1]})$. As $\ord(f_2)=m\geq 3$, the cosets $\la f_1\ra-f_2$ and $\la f_1\ra+f_2$ are disjoint, meaning the same possibility must occur for every block decomposition $S=W_0\bdot \ldots\bdot W_{2m-2}$. Thus either Claim A.3(a) holds for every such block decomposition, or else Claim A.3(b) does, establishing the first part of the claim.

Now assume $S^*=f_1^{[m-1]}\bdot f_2^{[m-1]}\bdot (f_1+f_2)$ and that there is a block decomposition $S=W_0\bdot \ldots\bdot W_{2m-2}$ with $S_\sigma=S^*$ and $\sigma(W_j)=f_1+f_2$ for some $j\in I_k$. By the first part of the claim applied using the basis $(f_1,f_2)$, either Claim A.3(a) or Claim A.3(b) holds for this block decomposition.
If Claim A.3(b) holds using the basis $(f_1,f_2)$, then $\sigma(W_i)=f_1$ for all $i\in I_k\setminus \{j\}$, and we have  $g_k\in \supp(W_j)$.
In such case, with respect to the basis $(f_2,f_1)$, all blocks $W_i$ with $i\in I_k$ have type (II), contrary to first part of claim applied using the basis $(f_2,f_1)$. Therefore  we must instead have  Claim A.3(a) holding using the basis $(f_1,f_2)$. In this case, there is a unique $i_0\in I_k$ with $\sigma(W_{i_0})=f_1$, and we have $g_k\in \supp(W_{i_0})$. With respect to the basis $(f_2,f_1)$, both $W_{i_0}$ and $W_j$ are type (II) blocks. Thus if we apply the  first part of the claim instead  using the basis $(f_2,f_1)$, we find that Claim A.3(a) must again hold,  there is  a unique $j_0\in I_k$ with  $\sigma(W_{j_0})=f_2$, and  $g_k\in \supp(W_{j_0})$. However, since $\sigma(W_{j_0})=f_2\neq f_1=\sigma(W_{i_0})$, we have $i_0\neq j_0$, making it impossible for term $g_k$ (which is the unique term of $S_k$ whose $\psi$ value differs from all other terms in $S_k$) to be contained in both $\supp(W_{i_0})$ and $\supp(W_{j_0})$. So we instead conclude $\sigma(W_j)\neq f_1+f_2$ for all $j\in I_k$, completing the claim.
\end{proof}

\subsection*{Claim E} For each $k\in [1,3]$, $\psi$ is constant on $S_k$.

\begin{proof}
In view of Claim D, for each $k\in [1,3]$, there is some $g_k\in \supp(S_k)$ such that $\psi$ is constant on $S_k\bdot g_k^{[-1]}$. By choosing the representatives $e_1$, $e_2$ and $e_3=xe_1+e_2$ appropriately, we can assume
$$\psi(g)=0\mbox{  for all $g\in \supp(S_1\bdot S_2\bdot g_1^{[-1]}\bdot g_2^{[-1]})$}\quad\und\quad \psi(g)=\alpha\mbox{ for all $g\in \supp(S_3\bdot g_3^{[-1]})$}.$$
Let $J\subseteq [1,3]$ be those indices $k$ such that $\psi$ is not constant on $S_k$ (so $k\in J$ means $\psi(g_k)\neq 0$ if $k\in [0,1]$ and $\psi(g_k)\neq \alpha$ if $k=3$). Assume by contradiction that $J$ is nonempty. In view of Claim C, there is a block decomposition $S=W_0\bdot\ldots\bdot W_{2m-2}$ with at least one of $W_0$ or $W_1$ having type (I), and Claim D ensures that either Claim A.3(a) or Claim A.3(b) holds for this block decomposition. The conditions given in Claim A.3(a) as well as  Claim A.3(b) ensure that if $k\in J$ with $g_k$ contained in the block $W_i$ and $g$ any other  term of $S_k$ contained in a distinct block $W_j$, then performing a $(\overline e_k,\overline e_k)$-swap using $g_k$ and $g$ results in a new block decomposition whose associated sequence remains unchanged and equal to $S^*$ with the types on the blocks $W_i$ and $W_j$ swapping.
If $W_0$ and $W_1$ both have type (I), then Claim A.3 ensures that all blocks containing one of the terms $g_k$, for $k\in J$, must have type (II). As $J$ is nonempty, we can thus perform a type $(\overline e_k,\overline e_k)$-swap between $g_k$ and a term $g$ from $W_0$ with $\varphi(g)=\overline e_k$,
  resulting in a new block decomposition with the type of $W_0$ changing from (I) to (II). This allows us to w.l.o.g. assume one of  $W_0$ and $W_1$ has type (I) and the other has type (II). Then any $k\in I_k$  having a unique block of type (I) that contains terms from $S_k$ or having a unique block of type (II) that contains terms from $S_k$ must have that block being either $W_0$ or $W_1$.  In consequence, it follows from Claim A.3 that  $g_k\in \supp(W_0\bdot W_1)$ for all $k\in J$, and by performing type $(\overline e_k,\overline e_k)$-swaps between $W_0$ and $W_1$, we can then also assume $g_k\in \supp(W_0)$ for all $k\in J$, in which case Claim A.3 ensures $$\psi(g_k)=F\mbox{ for $k\in J\cap [0,1]$}\quad\und\quad\psi(g_k)=\alpha+F\mbox{  when $k=3\in J$},$$ where $F=\sigma(W_0)-\sigma(W_1)\in \la f_1\ra \pm f_2$ with $\pm=+$ when $W_0$ has type (II) and $\pm=-$ when $W_0$ has type (I).

\subsection*{Case E.1:} $W_0$ has type (II).

In this case, $F=(y_0-1)f_1+f_2\in \la f_1\ra+f_2$. If $S^*=f_1^{[m-1]}\bdot f_2^{[m-1]}\bdot (f_1+f_2)$, then Claim D ensures that $y_0=0$, in which case $F=-f_1+f_2$. Let $k\in J$. Since $g_k\in \supp(W_0)$, it follows that Claim A.3(b) holds with  $W_0$ the unique block having type (II) that contains a term from $S_k$. Thus, in view of Claim B, it follows that there is some $j\in I_k\setminus\{0,1\}$ with $W_j$ having type (I). Since $x,x^*\in [2,n-2]$, $\vp_{e_1}(W_0)=n-1>x$ and $\vp_{e_2}(W_0)=x^*>1$, it follows that we can perform a type $(\overline e_2+,\overline e_3)$-swap between $W_0$ and $W_1$ that does not use any of the terms $g_i$ with $i\in J$. Applying Lemma \ref{lemma-pertub-I} or Lemma \ref{lemma-pertub-II(weak)} to the resulting weak block decomposition (and using that neither $\sigma(W_1)$ nor $\sigma(W_2)$ can equal $f_1+f_2$ if $S^*=f_1^{[m-1]}\bdot f_2^{[m-1]}\bdot (f_1+f_2)$ by Claim D), we find that $$-\alpha=(xe_1+e_2)-(e_3+\alpha)\in \{0,F\}.$$
Likewise, Since $x,n-x^*\in [2,n-2]$, $\vp_{e_1}(W_0)=n-1>n-x$ and $\vp_{e_3}(W_0)=n-x^*>1$, it follows that we can perform a type $(\overline e_3+,\overline e_2)$-swap between $W_0$ and $W_1$ that does not use any of the terms $g_i$ with $i\in J$ and then use  Lemma \ref{lemma-pertub-I} or Lemma \ref{lemma-pertub-II(weak)} to conclude that $$ne_1+\alpha=((n-x)e_1+e_3+\alpha)-e_2\in \{0,F\}.$$
We can also perform a type $(\overline e_k,\overline e_k)$-swap between $W_0$ and $W_j$ using $g_k$ and any $g\in \supp(W_j)$ resulting in a new block decomposition $S=W'_0\bdot\ldots\bdot W'_{2m-2}$ with $W'_0=W_0\bdot g_k^{[-1]}\bdot g$, $W'_j=W_j\bdot g^{[-1]}\bdot g_k$ and $W'_i=W_i$ for $i\neq 0,j$. As before, the associated sequence $S'_\sigma=S_\sigma=S^*$ remains fixed with the types of $W_0$ and $W_j$ swapping, so $W'_0$ has type (I) and $W'_j$ has type (II). But now, performing a type $(\overline e_2+,\overline e_3)$-swap between the type (I) blocks $W'_0$ and $W'_j$ and applying Lemma \ref{lemma-pertub-I} or Lemma \ref{lemma-pertub-II(weak)} to the resulting weak block decomposition implies that either $-\alpha=(xe_1+e_2)-(e_3+\alpha)=0$ or else $S^*=f_1^{[m-1]}\bdot f_2^{[m-1]}\bdot (f_1+f_2)$ with $-\alpha=(xe_1+e_2)-(e_3+\alpha)\in \la f_2\ra$. However, in the latter case, we have $F=-f_1+f_2$ with $-\alpha\in \{0,F\}$, which combines with $-\alpha\in \la f_2\ra$ to also yield $\alpha=0$. Thus $$\alpha=0$$ holds in all cases.  Likewise,
performing a type $(\overline e_3+,\overline e_2)$-swap between the type (I) blocks $W'_0$ and $W'_j$ and applying Lemma \ref{lemma-pertub-I} or Lemma \ref{lemma-pertub-II(weak)} to the resulting weak block decomposition implies that either $ne_1+\alpha=((n-x)e_1+e_3+\alpha)-e_2=0$ or else $S^*=f_1^{[m-1]}\bdot f_2^{[m-1]}\bdot (f_1+f_2)$ with $ne_1+\alpha\in \la f_2\ra$. However, in the latter case, we have $F=-f_1+f_2$ with $ne_1+\alpha\in \{0,F\}$, which combines with $ne_1+\alpha\in \la f_2\ra$ to also yield $ne_1+\alpha=0$. Thus $ne_1+\alpha=0$ holds in all cases, and since $\alpha=0$, we are left to conclude that $ne_1=0$. However, since $\psi(g)=0$ for all but at most one term of $S_1$ (namely $g_1$), it follows from Claim B that $\vp_{e_1}(S)\geq |S_1|-1\geq (|I_1|-2)n+n-1\geq 2n-1\geq n$, which means $S$ contains the proper zero-sum subsequence $e_1^{[n]}$, contradicting that $S$ is a minimal zero-sum sequence, which completes Case E.1.

\subsection*{Case E.2:} $W_0$ has type (I).

In this case, $F=(1-y_1)f_1-f_2\in \la f_1\ra-f_2$. If $S^*=f_1^{[m-1]}\bdot f_2^{[m-1]}\bdot (f_1+f_2)$, then Claim D ensures that $y_1=0$, in which case $F=f_1-f_2$. Let $k\in J$. Arguing as in Case E.1, we can perform a type $(\overline e_3+,\overline e_2)$-swap  and a   type $(\overline e_2+,\overline e_3)$-swap between $W_0$ and $W_1$, neither of which uses any of the terms $g_i$ with $i\in J$, and then use  Lemma \ref{lemma-pertub-I} or Lemma \ref{lemma-pertub-II(weak)} to conclude that $$-\alpha=(xe_1+e_2)-(e_3+\alpha)\in \{0,F\}\quad\und\quad ne_1+\alpha=((n-x)e_1+e_3+\alpha)-e_2\in \{0,F\}.$$ If $\alpha=0$ and $ne_1+\alpha=0$, then $ne_1=0$ and we obtain the same contradiction that completed Case E.1.

If $\alpha=-F$ and $ne_1+\alpha=0$, or if $\alpha=0$ and $ne_1+\alpha=F$,  then $ne_1=F=(1-y_1)f_1-f_2$. If $\alpha=-F$ and $ne_1+\alpha=F$, then $ne_1=2F=2(1-y_1)f_1-2f_2$.
In either case,  in view of Claim B, we have some $j\in I_1\setminus \{0,1\}$, so $\varphi(W_j)=\overline e_1^{[n]}$, and  since $\psi$ is constant and equal to $0$ on $S_1\bdot g_1^{[-1]}$, and since either $g_1\in \supp(W_0)$ or $\psi(g_1)=0$, we actually have $W_j=e_1^{[n]}$. Hence $ne_1=\sigma(W_j)\in \{f_1, y_jf_1+f_2\}$, both of which contradict $ne_1=(1-y_1)f_1-f_2$ as well as $ne_1=2(1-y_1)f_1-2f_2$ since $\ord(f_2)=m\geq 4$. As this exhausts all possibilities for $\alpha$ and $ne_1+\alpha$, Case E.2,  and also the proof of Claim E, is now complete.
\end{proof}

In view of Claim E, we can choose the representatives $e_1$, $e_2$ and $e_3$ such that
$$\psi(g)=0\mbox{ for all $g\in \supp(S_1\bdot S_2)$}\quad\und\quad \psi(g)=\alpha\mbox{ for all $g\in \supp(S_3)$}.$$ Thus all terms in the sequence $S_1$ are equal to $e_1$, all terms in $S_2$ are equal to $e_2$, and all terms in $S_3$ are equal to $e_3+\alpha=xe_2+e_1+\alpha$.
In view of Claim C, we can assume there is a block decomposition
$S=W_0\bdot\ldots\bdot W_{2m-2}$ such that at least one of  $W_0$ or $W_1$ has type (I). This allows us to divide the remainder of the proof into two cases.

\subsection*{Case 1:} $S_\sigma=f_1^{[m-1]}\bdot f_2^{[m-1]}\bdot (f_1+f_2)$ with $\sigma(W_0),\sigma(W_1)\in \{f_1,f_1+f_2\}$.

\begin{proof} Since $\sigma(W_0),\sigma(W_1)\in \{f_1,f_1+f_2\}$,
performing $(e_2+,e_3)$- and $(e_3+,e_2)$-swaps between the blocks $W_0$ and $W_1$ and applying Lemma \ref{lemma-pertub-II(weak)} yields  $\alpha, ne_1\in \la f_2\ra$.
Any block $W_j$ with $j\in I_1\setminus \{0,1\}$ has $ne_1=\sigma(W_j)\in \{f_1,f_1+f_2,f_2\}$, and since $ne_1\in \la f_2\ra$, all such blocks (of which there is at least one by Claim B) must have $ne_1=\sigma(W_j)=f_2$. In every block $W_j$ with $\sigma(W_j)=f_2$ has $W_j=e_1^{[n]}$, then $\vp_{e_1}(S)=(|I_1|-1)n=mn$, so that $e_1^{[mn]}$ is a proper zero-sum subsequence of $S$, contradicting that $S$ is a minimal zero-sum sequence. Therefore we instead conclude that there is some block $W_j$ with $\sigma(W_j)=f_2$ and either $W_j=e_2^{[n]}$ or $W_j=(e_3+\alpha)^{[n]}$. Thus either $ne_2=f_2$ or $n(e_3+\alpha)=xne_1+ne_2+n\alpha=f_2$, and both cases imply $ne_1,ne_2,n(e_3+\alpha)=xne_1+ne_2+n\alpha\in \la f_2\ra$ in view of $\alpha\in \la f_2\ra$. As a result, letting  $i\in [1,2m-2]$ be a block with $\sigma(W_i)=f_1+f_2$ or $f_1$ (which exists since $m\geq 3$), we find that $f_1+f_2=\sigma(W_i)\in \{ne_1,ne_2,n(e_3+\alpha)\}\subseteq \la f_2\ra$ or $f_1=\sigma(W_i)\in \{ne_1,ne_2,n(e_3+\alpha)\}\subseteq \la f_2\ra$, neither of which is not possible.

\end{proof}

 \subsection*{Case 2:} Either $f_1$ is the unique term with multiplicity $m-1$ in $S_\sigma$, or else $S_\sigma=f_1^{[m-1]}\bdot f_2^{[m-1]}\bdot (f_1+f_2)$ with $\{\sigma(W_0),\sigma(W_1)\}=\{f_1,f_2\}$.

\begin{proof}
If $f_1$ is the unique term with multiplicity $m-1$ in $S_\sigma$ and $W_0$ and $W_1$ both have type (I), then performing $(e_2+,e_3)$- and $(e_3+,e_2)$-swaps between the type (I) blocks $W_0$ and $W_1$ and applying Lemma \ref{lemma-pertub-I} yields $-\alpha=(xe_1+e_2)-(e_3+\alpha)=0$ and $ne_1+\alpha=((n-x)e_1+e_3+\alpha)-e_2=0$, implying $ne_1=-\alpha=0$. But then $e_1^{[n]}$ is a proper zero-sum subsequence of $S$, contradicting that $S$ is a minimal zero-sum sequence.
Therefore we instead conclude that one of the blocks $W_0$ or $W_1$ has type (I) and the other has type (II). Moreover, $\{\sigma(W_0),\sigma(W_1)\}=\{f_1,f_2\}$ when $S_\sigma=f_1^{[m-1]}\bdot f_2^{[m-1]}\bdot (f_1+f_2)$ by case hypothesis.
As a result, performing $(e_2+,e_3)$- and $(e_3+,e_2)$-swaps between the  blocks $W_0$ and $W_1$ and applying Lemma \ref{lemma-pertub-I} or Lemma \ref{lemma-pertub-II(weak)} yields $$-\alpha=(xe_1+e_2)-(e_3+\alpha)\in \{0,F\}\quad\und\quad ne_1+\alpha=((n-x)e_1+e_3+\alpha)-e_2\in \{0,F\},$$ where $F=\sigma(W_0)-\sigma(W_1)$.
Note $F=(1-y_1)f_1-f_2$ if $W_0$ has type (I), and $F=(y_0-1)f_1+f_2$ if $W_0$ has type (II).

If $\alpha=0$ and $ne_1+\alpha=0$, then $ne_1=0$. If $\alpha=0$ and $ne_1+\alpha=F$, then $ne_1=F$. If  $\alpha=-F$ and $ne_1+\alpha=0$, then $ne_1=F$. If  $\alpha=-F$ and $ne_1+\alpha=F$, then $ne_1=2F$. If $ne_1=0$, then $e_1^{[n]}$ is a proper zero-sum subsequence of $S$, contradicting that $S$ is a minimal zero-sum sequence. We conclude that $ne_1\in \{F,2F\}$. In view of Claim B, there is some block $W_j$ with $W_j=e_1^{[n]}$, forcing $ne_1=\sigma(W_j)\in \{f_1,y_jf_1+f_2\}$. Since we also have $ne_1\in \{F,2F\}$ with  $F=(1-y_1)f_1-f_2$ or $F=(y_0-1)f_1+f_2$, we find this is only possible (in view of  $\ord(f_2)=m\geq 4$) if $W_0$ and $W_j$ both have type (II) with  $$\sigma(W_0)=y_0f_1+f_2\quad\und\quad ne_1=\sigma(W_j)=F=(y_0-1)f_1+f_2=y_jf_1+f_2.$$

 If  $S_\sigma=f_1^{[m-1]}\bdot f_2^{[m-1]}\bdot (f_1+f_2)$, then our case hypothesis ensures that $y_0=0$ and $-f_1+f_2=(y_0-1)f_1+f_2=\sigma(W_j)\in \{f_2,f_1+f_2\}$, which is not possible in view of $\ord(f_1)=m\geq 3$. Therefore it remains to consider the case when $f_1$ is the unique term with multiplicity $m-1$ in $S_\sigma$.

 If $W_i$ is a block with $i\in [1,2m-2]$, then we have $W_i=e_1^{[n]}$ or $W_i=e_2^{[n]}$ or $W_i=(e_3+\alpha)^{[n]}$. If $W_i=e_1^{[n]}$, then $\sigma(W_i)=ne_1=(y_0-1)f_1+f_2$ as shown above. If every type (II) block $W_i$ with $i\in [1,2m-2]$ has $W_i=e_1^{[n]}$, then $\vp_{e_1}(S)=(|I_1|-1)m=mn$, in which case $e_1^{[mn]}$ is a proper zero-sum subsequence of $S$, contradicting that $S$ is a minimal zero-sum sequence. Therefore there must be some $k\in [1,2m-2]$ with $W_{k}$ a type (II) block and either  $W_{k}=e_2^{[n]}$ or $W_{k}=(e_3+\alpha)^{[n]}$. Thus either $ne_2=\sigma(W_k)=y_kf_1+f_2$ or $ne_3+n\alpha=\sigma(W_k)=y_kf_1+f_2$.
 Swapping the role of $\overline e_2$ and $\overline e_3$ if need be, we can w.l.o.g.  assume the latter occurs, so $ne_3+n\alpha=\sigma(W_k)=y_kf_1+f_2$.
 This ensures that every type (I) block $W_i$ with $i\in [1,2m-2]$ must have $W_i=e_2^{[n]}$ with $ne_2=\sigma(W_i)=f_1$ (since $ne_3+n\alpha=y_kf_1+f_2\neq f_1$ and $ne_1=(y_0-1)f_1+f_2\neq f_1$). Performing a $(e_2+,e_3)$-swap between the type (II) blocks $W_0$ and $W_k$ and applying Lemma \ref{lemma-pertub-I}, we find that $-\alpha=(xe_1+e_2)-(e_3+\alpha)\in \la f_1\ra$. Combined with the previous conclusion that $-\alpha\in \{0,F\}$ with $F=(y_0-1)f_1+f_2$, it follows that $\alpha=0$. Summarizing these conclusions, we now know
 \begin{align*}
 ne_1=(y_0-1)f_1+f_2,\quad ne_2=f_1,\quad\alpha=0,\quad \und\quad ne_3=y_kf_1+f_2.
 \end{align*}
Thus $y_kf_1+f_2=ne_3=xne_1+ne_2=(x(y_0-1)+1)f_1+xf_2$, which forces $x\equiv 1\mod m$ and $y_k\equiv y_0\mod m$. As a result, \emph{every} type (II) block $W_i$  for $i\in [0, 2m-2]$ has $\sigma(W_i)\in \{\sigma(W_0), \,ne_1,\, ne_3\}=\{y_0f_1+f_2, (y_0-1)f_1+f_2\}$.
Letting $t\in [0,m]$ be the number of type (II) blocks $W_i$ with $\sigma(W_i)=(y_0-1)f_1+f_2$, so that
$m-t$ is the number of type (II) blocks $W_i$ with $\sigma(W_i)=y_0f_1+f_2$, we have $$t(y_0-1)+(m-t)y_0\equiv \Summ{i\in B_{(II)}}y_i\equiv 1\mod m,$$ whence $t=m-1$, contradicting that $f_1$ is the unique term of $S_\sigma$ having multiplicity $m-1$, 
which completes the case and proof.
\end{proof}
\end{proof}

\end{document}